\documentclass[12pt]{article}
\usepackage{lmodern}
\usepackage{graphicx}
\usepackage{caption}
\usepackage{mathtools}
\usepackage{amssymb}
\usepackage{upgreek}
\usepackage{dsfont}
\usepackage{cancel}
\usepackage{stmaryrd} 
\usepackage{ulem}
\usepackage{stackengine}
\usepackage{blkarray} 

\usepackage{time} 

\usepackage{fourier-orns} 

\usepackage{stackengine}
\usepackage{scalerel}

\usepackage{tikz}
\usepackage{textcomp}
\usepackage[margin=.8in]{geometry}

\usepackage{amsmath}
\usepackage{amssymb}
\usepackage{amsthm}

\usepackage{soul}

\usepackage[hidelinks]{hyperref}

\usepackage[left]{lineno} 

\usepackage{enumerate}
\usepackage[shortlabels]{enumitem} 	
\setlist[itemize]{noitemsep} 			

 \usepackage{float}

\newtheorem{theorem}{Theorem}[section]
\newtheorem{lemma}[theorem]{Lemma}
\newtheorem{proposition}[theorem]{Proposition}
\theoremstyle{definition}
\newtheorem{sublemma}[theorem]{Sublemma}
\newtheorem{definition}[theorem]{Definition}

\newtheorem{corollary}[theorem]{Corollary}

\theoremstyle{remark}
\newtheorem{remark}[theorem]{Remark}

\numberwithin{equation}{section}

\newcommand{\T}{\mathbb{T} }
 \newcommand{\R}{\mathbb{R}}
 \newcommand{\C}{\mathbb{C}}
  \newcommand{\N}{\mathbb{N}}
 \newcommand{\Z}{\mathbb{Z}}
   
      \newcommand{\BB}{\mathcal{B}}
  
    \newcommand{\MM}{\mathcal{M}}
     \newcommand{\CC}{\mathcal{C}}

          \newcommand{\Tr}{\operatorname{Tr}}
                    
\DeclareMathOperator{\Span}{span}
\DeclareMathOperator{\proj}{proj}

\usetikzlibrary{arrows}

\author{M.Chatal, C.Chavaudret, L.H.Eliasson}
\title{Real almost reducibility of differentiable real quasi-periodic cocycles}
\date{}

\begin{document}

\maketitle

\begin{abstract}
\noindent We prove that infinitely differentiable almost reducible quasi-periodic cocycles, under a Diophantine condition on the frequency vector, are almost reducible to a real constant cocycle with a real conjugation, up to a period doubling.\end{abstract}


\section{Introduction}
Let $\omega \in \R^d$ a rationally independent vector, i.e. a vector satisfying
\[ \langle k, \omega \rangle \neq 0, \quad \forall k \in \Z^d\setminus\{0\}, \]
and let $A : \T^d \rightarrow \mathcal{M}(n, \C)$ of class $\CC^\infty$, where $\T^d=\R^d/\Z^d$. The quasiperiodic cocycle associated to $A$ is the map (of class $\CC^\infty$)  $X_{\omega, A} : \R \times \T^d \rightarrow Gl(n, \C)$  which is solution of  
\[
\left\{\begin{array}{l}
\frac{d}{dt}X_{\omega, A}^t(\theta) = A(\theta + t\omega)X^t_{\omega, A}(\theta) \\
X^0_{\omega, A}(\theta) = I.
\end{array}\right.
\]

\begin{remark}
In this paper, all functions and mappings will, unless otherwise specified,  be of class $\CC^\infty$.
\end{remark}

\medskip

\noindent We will say that the cocycle $X_{\omega, A}$  is \textit{real} if $A$ is a real valued map, and that it is  \textit{constant} if $A$ is a constant map.

\noindent A cocycle $X_{\omega, A}$ is \textit{conjugated} to a cocycle $X_{\omega, B}$ if and only if there exists a map $Z: \T^d \rightarrow Gl(n, \C)$  such that 
\begin{equation}\label{eq1.1}X^t_{\omega, A}(\theta) = Z(\theta + t\omega)X^t_{\omega, B} Z(\theta)^{-1}\quad \forall (t, \theta) \in \R \times \T^d.   \end{equation}
The mapping $Z: \T^d \rightarrow Gl(n, \C)$ is a \textit{conjugation} between $X_{\omega, A}$ and $X_{\omega, B}$. It satisfies the condition
\begin{equation}\label{eq1.2}
 \partial_\omega Z (\theta) = A(\theta)Z(\theta) - Z(\theta)B(\theta)\quad \forall  \theta \in  \T^d , \end{equation}
 where
 \[ \partial_\omega Z(\theta) = \frac{d}{dt}Z(\theta + t \omega)_{\vert t=0},\]
which is equivalent to  (\ref{eq1.1}).

\noindent A cocycle  is \textit{reducible} if and only if it is conjugated to a constant cocycle.
A cocycle  is \textit{real reducible} if and only if it is real and conjugated to a constant cocycle by a real conjugation.

\medskip

\noindent A natural question is whether a real and reducible cocycle is real reducible. The answer is yes modulo a  ``\textit{period-doubling}'':

\begin{theorem}\label{mainthm1}
        If $X_{\omega, A}$ is a real and reducible cocycle, then $X_{\frac\omega2, A_2}$ is real reducible, where
$$A_2(\theta)=A(2\theta)\quad \forall \theta\in\T^d.$$
\end{theorem}

\noindent Hence, there exist a map $Z: \T^d \rightarrow Gl(n, \R)$ and a constant matrix $B\in gl(n,\R)$ such that
\begin{equation}\label{eq1.3}
 \partial_{\frac\omega2} Z (\theta) = A(2\theta)Z(\theta) - Z(\theta)B \quad \forall  \theta \in  \T^d . \end{equation}
 
\noindent If we denote
 $$W(2\theta)=Z(\theta),$$
then  (\ref{eq1.3}) says that
$$ \partial_{\omega} W (\theta) = A(\theta)W(\theta) - W(\theta)B\quad \forall  \theta \in 2\T^d.$$
This looks very much like real reducibility of $X_{\omega, A}$, with the difference that $W$ is not defined on $\T^d$, but only on  the $2^d$-fold covering  $\R^d/(2\Z)^d$  of $\T^d$ --  this ``period-doubling''  cannot be avoided in general.

Theorem \ref{mainthm1} was proven in the article \cite{C10}, which also contains several other results of similar nature. 

\medskip

In this paper we shall discuss a similar result in the framework of almost reducible cocycles, i.e. cocycles that can be conjugated arbitrarily close to constant cocycles.
There is no canonical meaning of ``arbitrarily close'' and we shall use a pretty stringent formulation.

A cocycle $X_{\omega, A}$ is \textit{almost reducible } if and only if there exist sequences of maps $Z_j: \T^d \rightarrow Gl(n, \C)$, $F_j : \T^d \rightarrow gl(n, \C)$  and a sequence of matrices $B_j\in gl(n,\C)$ such that 
\begin{equation}\label{eq1.4}
 \partial_\omega Z_j (\theta) = A(\theta)Z_j(\theta) - Z_j(\theta)\big( B_j+F_j(\theta)\big) \quad \forall  \theta \in  \T^d  \end{equation}
 with
 \begin{equation} \lim_{j \rightarrow + \infty} \Vert Z_j^{\pm 1} \Vert_{\CC^r}^m \Vert F_j \Vert_{\CC^r} = 0, \quad \forall r, m \in \N. \ \ 
 \footnote { this formulation denotes the condition:  $Z_j^{ +1}=Z_j$
 }
  \label{conditionconvergence}\end{equation}
  
We will say that a real cocycle $X_{\omega, A}$ is \textit{real almost reducible} if and only if it is almost reducible with a sequence of real-valued maps $Z_j$ and a sequence of
real matrices $B_j$ satisfying \eqref{eq1.4} and \eqref{conditionconvergence} (then $F_j$ is automatically real).

\medskip

In this paper we shall prove:

\begin{theorem} If $X_{\omega, A}$ is a real cocycle which is almost reducible and  $\omega$ is Diophantine, then $X_{\frac{\omega}{2}, A_2}$ is real almost reducible.
\end{theorem}
This means that there exist sequences of maps $Z_j: \T^d \rightarrow Gl(n, \R)$, $F_j : \T^d \rightarrow gl(n, \R)$  and a sequence of matrices $B_j\in gl(n,\R)$ such that
$$
 \partial_{\frac{\omega}2} Z_j (\theta) = A(2\theta)Z_j(\theta) - Z_j(\theta)( B_j+F_j(\theta)) \quad \forall  \theta \in  \T^d$$
 and satisfying \eqref{conditionconvergence}.

There is an important difference compared to the reducible case. In the reducible case, no arithmetical condition on the frequency vector $\omega$ is needed, but in the almost reducible case the proof requires some such conditions. Indeed, we need to ensure that the primitive of a quasi-periodic function
$$\R\ni t\mapsto f(t\omega)$$
is quasi-periodic and smooth, and this requires arithmetical conditions on the frequency vector $\omega$. We have no idea if the theorem remains true without such conditions.

\bigskip

Reducibility and almost reducibility of quasiperiodic cocycles are important properties to understand the behaviour of cocycles. See for example   \cite{E92}  and \cite{Y18} for applications to one-dimensional quasi-periodic Schr\"odinger operators, and   \cite{E01} for applications to  quasi-periodic cocycles 
on $SO(3,\R)$.

The notion of almost reducibility is strictly weaker than that of reducibility
\footnote{ except when $n=1$ or $d=1$}.  For example, in the analytic perturbative case, under arithmetical assumptions on the frequency vector, there are constant cocycles, all of whose perturbations are  almost reducible but, generically, not reducible --  see for example  \cite{E92} and   \cite{E01}.

Both  reducibility and almost reducibility are very much perturbative phenomena\footnote{ except when $n=1$ or $d=1$}. Most results are available in analytic or ultradifferentiable category (see \cite{E92},\cite{E01},\cite{HY12},\cite{CC23}...), much fewer
in class $\CC^\infty$  --  for  a result in $\CC^\infty$ see for example \cite{FK09}.

Perturbative reducibility results require arithmetical conditions on the frequency vector. A Diophantine condition is most often used but it can be relaxed to a Brjuno-Rüssmann condition (see \cite{CM12},\cite{BCL21}...).

Without arithmetical conditions, weaker properties 
\footnote{such as rotational reducibility and almost rotational reducibility}
have been proven using renormalization techniques (see for example \cite{AK06}). These results are for the moment very much restricted to two frequencies $\omega=(\omega_1,\omega_2)$.

\subsection{A word about the proof}

It is  pretty forward to show that if a real cocycle can be conjugated to a real matrix by a complex conjugation, then it can be conjugated by a real conjugation. Therefore it suffices to prove that a real cocycle can be conjugated to a real matrix.

In the article \cite{C10}, (complex)  invariant  subbundles of a real cocycle are used to construct real invariant  subbundles. In this paper we give another proof of this. Indeed here we prove that a complex matrix to which a real reducible cocycle can be conjugated, has the same spectral properties (i.e. Jordan normal form)  as a real matrix -- it can therefore be conjugated to a real matrix. We prove this in section \ref{section2}. We generalize then this approach to almost reducible cocycles, but there are several complications. 

We would like to prove that if a real cocycle can be almost conjugated to a sequence of real matrices by complex conjugations, then it can be conjugated by real conjugations. This may be true, but we have only been able to prove this under a Diophantine condition on the frequency vector $\omega$  --  see section \ref{section3}.

We can always conjugate a matrix to Jordan normal form but we have no control on the conjugation.  In the reducible case this gives no problem, but in the almost reducible case it does since we need to ensure convergence condition \eqref{conditionconvergence}. This problem is treated in section \ref{section4}. Finally a result about almost reducibility to a Jordan normal form is given in section \ref{section4bis}.

In section \ref{section5} we analyse the spectral properties of a complex matrix to which a real cocycle is almost reduced. We show that, up to a 
sufficiently small perturbation, it has the same spectral properties as a real matrix.

Finally we put the results from sections \ref{section4bis} and \ref{section5} together and show that the estimates obtained are good enough to guarantee almost reducibility to real cocycles.

\subsection{Notations}

For any set $X$, we denote by $\# X$ its cardinality.

\medskip

\noindent 
For any $n\times n$-matrix $A$ we denote by $\sigma(A)$ its spectrum, that is to say the subset of $\C$ consisting of the eigenvalues of $A$. Clearly $\#\sigma(A)\le n$. Given $\alpha \in \sigma(A)$, we denote by $\text{mult}(\alpha)$ its (algebraic) multiplicity.

\medskip

\noindent 
Since all norms on $gl(n,\C)$ are equivalent, the definition of almost reducibility does not depend on the choice of matrix norms. We shall usually, unless otherwise said,
use the operator norm denoted by $\Vert \cdot \Vert$, but any other norm on $gl(n,\C)$  would do.

\noindent For a vector $k=(k_1,\dots k_d)\in\mathbb{R}^d$, denote by $|k|$ its $l^1$ norm: 
$\vert k \vert = \sum_{i=1}^d \vert  k_i \vert$.

\noindent 
As for the function norms, they are the usual:
\[ \Vert A \Vert_{\CC^0} = \displaystyle \sup_{\theta} \Vert A(\theta) \Vert\]
and for all $r \in \N$, 
\[ \Vert A \Vert_{\CC^r} = \max \{ \Vert \partial^\alpha A \Vert_{\CC^0} : \alpha \in \N^d, \vert \alpha \vert \leq r \}.\]
The norm $\Vert . \Vert_{\CC^r}$ is a complete norm on the space of (matrix-valued) $\CC^r$-functions.

\noindent 
Let us also recall  two  inequalities which we shall use frequently:
\begin{equation}\label{1}
\Vert AB\Vert_{\CC^r} \leq C_r \Vert A \Vert_{\CC^r}\Vert B \Vert_{\CC^r}, \quad \forall A, B \in \CC^r(\T^d, gl(n,\C))
\end{equation}

\begin{equation}\label{2}
\Vert A^{-1}\Vert_{\CC^r} \leq C_r \Vert A^{-1} \Vert_{\CC^0}^{r+1}\Vert A \Vert_{\CC^r}^r, \quad \forall A \in \CC^r(\T^d, Gl(n,\C))
\end{equation}
\noindent where $C_r$ is a constant which depends on $r$.

\medskip

\noindent 
Let us finally recall the Diophantine condition. We  say that $\omega \in \mathcal{DC}(\kappa, \tau)$ (for some $\kappa >0$ and $\tau > d-1$) if and only if
 \begin{equation}\label{DC} \vert \langle k, \omega \rangle \vert \geq \frac{\kappa}{\vert k \vert^{\tau}}, \quad \forall k\in \Z^d\setminus\{0\}.\end{equation}

\section{Real reducibility}\label{section2}
In this section, we prove a real reducibility proposition. This result has been already proved in \cite{C10}, but the proof here is different, and will be useful to understand the real almost reducibility result later.
\begin{proposition} \label{redpreli}
If two real cocycles $X_{\omega, A}$ and $X_{\omega, B}$ of dimension $n$ are conjugated, then they are conjugated by a real conjugation.
\end{proposition}
\begin{proof}
Let $Z : \T^d \rightarrow Gl(n,\C)$ a conjugation between $X_{\omega, A}$ and $X_{\omega, B}$ 
\[\partial_\omega Z = A Z - ZB.\]
The polynomial \[ \det (\Re Z(0) + \lambda \Im Z(0))\]
is of degree $n$ and is not the zero polynomial because it doesn't vanish when $\lambda = i$ (since $Z(0) = \Re Z(0) + i \Im Z(0)$ is invertible) and thus it has at most $n$ zeros. Choose $\lambda \in \R$ such that $\Re Z(0) + \lambda \Im Z(0)$ is invertible and let 
\[ W(\theta) = \Re Z(\theta) + \lambda \Im Z(\theta), \]
then
\[ X^{t}_{\omega, A}(\theta) W(\theta) = W(\theta + t\omega)X^{t}_{\omega, B}(\theta), \quad \theta \in \T^d. \]
Moreover, if there exists $\theta \in \T^d$ such that $\det W(\theta) = 0$, then the previous relations imply $\det W(\theta + t\omega) = 0$ for all $t\in \R$.  Since $\omega$ is rationally independent, $\{ [t\omega] ; t\in \R \}$ is a dense set in $\T^d$ and the continuity of $\det W$  implies $\det W = 0$, which is impossible because we chose $\lambda$ real such that $\Re Z(0) + \lambda \Im Z(0)$ is invertible. Thus $W(\theta)$ is invertible for all $\theta \in \T^d$.
\end{proof} 

\begin{remark}
The real conjugation in the proposition \ref{redpreli} and the given conjugation have the same period: there is no period doubling.
    
\end{remark}

\begin{proposition} \label{reducprincipal}
Let 
$U : \T^d \rightarrow Gl(n,\C)$ continuous and let $B\in \mathcal{M}(n,\C)$ in Jordan normal form.
If
\begin{equation}
\partial_\omega U = BU - U \bar{B} \label{conjug}
\end{equation}
then there exist $W:\T^d \rightarrow Gl(n,\C)$ of class $\mathcal{C}^\infty$ and $B' \in \mathcal{M}(n, \R)$ such that, for all $\theta \in \T^d$
\[ \partial_{\frac{\omega}{2}}W(\theta) = BW(\theta) - W(\theta)B'.\]
\end{proposition}

\bigskip
\noindent 
\textbf{Remark:} The map $W$ defined here is in fact a trigonometric polynomial.

\medskip
We postpone the proof of Proposition \ref{reducprincipal} after a few lemmas.

\noindent Denote
\[ \mathcal{M} = \left\{2i\pi \langle k , \omega \rangle, k \in \Z^d \right\}.\]
Let  $\sigma(B) = \{\alpha_1, \cdots, \alpha_l\}$ the spectrum of $B$. 
Relation \eqref{conjug} implies, denoting $B = diag(B_j)_{j = 1, ...,l}$, $\alpha_j$ the eigenvalue of a block $B_j$ and $U =(U^j_i)_{i,j = 1, ...., l}$
\begin{equation}\label{conjugaison-B-blocks} \partial_{\omega}U^j_i = B_i U^j_i - U^j_i \bar B_j. \end{equation}
\begin{lemma}\label{lemma1}
\begin{enumerate}
\item If $\alpha_i - \bar \alpha_j \notin \MM$, 
the block $U_i^j$ is zero. In particular, if $\alpha_i - \bar \alpha_j \notin \MM$ for a given $j$ and for all $i$, then $\det U = 0$.

\item Moreover, if $\alpha_i-\bar{\alpha}_j=2i\pi \langle k_{i,j},\omega\rangle$ for some $k_{i,j}\in \Z^d\backslash \{0\}$, then the only non zero Fourier mode of the block $U_i^j$ is indexed by $k_{i,j}$.
\end{enumerate}
\end{lemma}

\begin{proof}

The relation \eqref{conjug} implies, if $U_i^j = (u_{i',j'})$, denoting $s_i = \text{mult}(\alpha_i)$ and $s_j = \text{mult}(\alpha_j)$, for all $(i',j') \in \llbracket 1, s_i \rrbracket \times \llbracket 1, s_j \rrbracket$ (
letting $\delta_0 = \delta_{s_i} = 0$), 
\begin{equation}
\label{relcoefred} \partial_\omega u_{i',j'} = (\alpha_i - \bar \alpha_j)u_{i',j'} + \delta_i u_{i'+1,j'} - \delta_{j'-1}u_{i',j'-1},  
\end{equation}
with the $\delta_{i'}, \delta_{j'} \in \{0,1 \}$. 
Let $i,j$ such that $\alpha_i - \bar \alpha_j \notin \mathcal{M}$, then decomposing these equations with Fourier coefficients and solving the equations in the right order, this implies
\[ u_{i',j'} = 0\]
whenever $\alpha_{i'} = \alpha_i$ and $\alpha_{j'} = \alpha_j$.

Now, given $\alpha_j$, if $\alpha_j - \beta \notin \MM$ for all $\beta$, this implies that there is a zero column in $U$ and then $\det U = 0$ which is impossible by assumption on $U$.

\bigskip
Now assume that for some $i,j$, $\alpha_i-\bar \alpha_j = 2i\pi \langle k_{i,j},\omega\rangle$ where $k_{i,j}\in\mathbb{Z}^d$. If $\delta_{i'}=\delta_{j'-1}=0$, then \eqref{relcoefred} implies that the only non zero Fourier mode of $u_{i',j'}$ is indexed by $k_{i,j}$. Then recursively on $i',j'$ corresponding to the same two eigenvalues of $B$, one proves that the only non zero Fourier mode of the block $U_i^j$ is indexed by $k_{i,j}$.

\end{proof}

\begin{definition}

We will say that two complex numbers $\alpha_i$ and $\alpha_j$  are linked if and only if there exists $k_{i,j}\in \Z^d$ such that

\[ \alpha_i - \bar \alpha_j = 2i\pi \langle k_{ij}, \omega \rangle \]
(that is to say $\alpha_i - \bar \alpha_j \in \MM$). 
\end{definition}

\begin{remark}

The relation of being linked is symmetric, but neither reflexive nor transitive.
\end{remark}

\begin{definition}\label{chain-loop}
We call
\textit{chain of length $k-1$} a sequence $\alpha_{i_1}, \cdots, \alpha_{i_{k}}$ such that for all $j \in \{1, \cdots, k-1\}$, $\alpha_{i_j}$ is linked to $\alpha_{i_{j+1}}$.
If moreover $\alpha_{i_{k}}= \alpha_{i_1}$ and $k\geq 2$, we will say it is a \textit{loop of length $k-1$}.

We will say that two numbers $\alpha,\beta$ are chain-linked if there is a chain between $\alpha$ and $\beta$. 
\end{definition}

\begin{remark} \label{relations} 
The relation of being chain-linked is an equivalence relation on any set of complex numbers $\Gamma$ which satisfy 
\begin{equation}\label{star} \text{for all $x\in \Gamma$, there exists $y\in \Gamma$ such that $x$ and $y$ are linked} \end{equation} (this will be the case when we will consider the spectrum of our matrix $B$). Considering such a set $\Gamma \subseteq \C$, we shall denote by $[\alpha]$ the equivalence class of $\alpha\in \Gamma$.

\bigskip
Notice that if $\alpha,\beta$ are linked by a chain with even length, then $\alpha-\beta \in\MM$, and if they are linked by a chain with odd length, then $\alpha-\bar \beta \in \MM$.

\bigskip
Also, it is easy to notice that if there is a chain of length $3$ then the first and the last numbers are linked (simply write the resonances relations).
\end{remark}

\begin{sublemma}\label{souslemme}
Let $\Gamma \subseteq \C$ satisfying property $\eqref{star}$.  Given $\alpha \in \Gamma$, $\alpha$ is linked to itself if and only if $[\alpha]$ contains a loop of odd length.
\end{sublemma}

\begin{proof}
    If $\alpha$ is linked to itself, then there is a loop of length $1$ between $\alpha$ and itself.
    Suppose now there is a loop of odd length in $[\alpha]$. Denote $\alpha_1, \dots \alpha_{k}$ this loop, and suppose $\alpha = \alpha_1$. Then,
    \[\alpha_1 - \bar \alpha_1 = \alpha_1 - \bar \alpha_2 + \bar \alpha_2 - \alpha_3 + \dots + \alpha_{k}-\bar \alpha_1 \in \MM\]
    and then $\alpha_1$ is linked to itself. 
\end{proof}

Let us investigate the possible links between the eigenvalues of $B$. 

\begin{lemma}\label{relationsvp} Under the assumptions of Proposition \ref{reducprincipal},
$\sigma(B)$ satisfies property \eqref{star}. Moreover, for all $\alpha\in \sigma(B)$, if $[\alpha]$ does not contain any loop of odd length, then there is a partition of $[\alpha]=\Sigma_1\cup \Sigma_2$ where
\begin{itemize}
\item every element of $\Sigma_1$ is linked to every element of $\Sigma_2$,
\item the sum of multiplicities of the eigenvalues in $\Sigma_1$ equals the sum of multiplicities of the eigenvalues in $\Sigma_2$.
\end{itemize}
\end{lemma}

\begin{proof}

By the lemma \ref{lemma1}, the property \eqref{star} holds, and the blocks of $U$ relating different equivalence classes of eigenvalues are zero. Therefore, if $U$ has invertible values, the blocks of $U$ corresponding to an equivalence class of $\sigma(B)$ are invertible.

\bigskip 
Define the equivalence relation
\[ \alpha_i \sim \alpha_j \Leftrightarrow \text{there exists a chain of even length between } \alpha_i \text{ and } \alpha_j \text{ in } [\alpha].\]
From remark \ref{relations}, there are only 2 distinct equivalence classes $\Sigma_1 = \{\beta_1, \dots, \beta_r \}$ and $\Sigma_2 = \{\gamma_1, \cdots, \gamma_s \}$ for $\sim$. Moreover, two elements of the same equivalence class cannot be linked, otherwise we would have a loop of odd length, which is impossible by assumption.
Any element in $\Sigma_1$ is linked to all elements of $\Sigma_2$. Indeed let $\beta \in \Sigma_1$ and $\gamma\in \Sigma_2$, then $\beta $ and $\gamma$ are linked by a chain (because they both are in $[\alpha]$), now this chain has an odd length by definition of $\Sigma_1,\Sigma_2$, therefore they are linked.

Then from lemma \ref{lemma1}, the block $ \tilde U $ corresponding to $[\alpha]$ has the form 
\[ \tilde U = 
\begin{blockarray}{cccccccc}
& \beta_1  &\cdots& \beta_r & \gamma_1 & \cdots & \gamma_s \\
  \begin{block}{c(ccccccc)}
\beta_1 & 0 & \cdots & 0 & \ \ast & \cdots & \ast \\
\vdots & \vdots & \ddots  & \vdots & \vdots & \ddots & \vdots \\
\beta_r & 0 & \cdots & 0 &  \ast & \cdots & \ast \\
\gamma_1 & \ast & \cdots & \ast &  0 & \cdots & 0 \\
\vdots & \vdots & \ddots & \vdots &\vdots & \ddots & \vdots \\
\gamma_s & \ast & \cdots & \ast & 0 & \cdots & 0 \\
  \end{block}
\end{blockarray}
\]

\noindent 
(because the $\beta_i$ are not linked to each other, nor are the $\gamma_i$),
which is invertible only if $\sum_{i=1}^r \text{mult}(\beta_i) = \sum_{i=1}^s \text{mult}(\gamma_i)$.
\end{proof}

We can now prove proposition \ref{reducprincipal}:
\begin{proof}{[of the proposition \ref{reducprincipal}]}
\textbf{Construction of $W$.}
By lemma \ref{relationsvp}, $\sigma(B)$ satisfies the property  \eqref{star} therefore the relation of being chain linked is an equivalence.
Let $\mathcal{E} \subset \sigma(B)$ be an equivalence class for the relation of being chain linked. There are two cases:
\begin{enumerate}
\item If $\mathcal{E}$ contains a loop of odd length, then by sublemma \ref{souslemme}, for all $\alpha_i \in \mathcal{E}$,
\[\alpha_i - \bar \alpha_i = 2i\pi \langle k_{i}, \omega \rangle\] with $k_{i} \in \Z^d$ (therefore $\operatorname{Im}(\alpha_i)= \pi \langle k_i,\omega\rangle$). 
\item Otherwise, $\mathcal{E}$ does not contain odd loops, and from lemma \ref{relationsvp} write $\mathcal{E} = \Sigma_1 \cup \Sigma_2$ where $\Sigma_1,\Sigma_2$ are the equivalence classes for $\sim$ (here the index 1 or 2 is arbitrary). Choose arbitrarily $\alpha \in \Sigma_1$. Then from remark \ref{relations}, 
for all $\alpha_i \in \Sigma_1$, there exists 
$k_{i} \in \Z^d$
such that \[ \alpha_i - \alpha = 2i\pi \langle k_{i}, \omega \rangle. \] 
Also, for all $\alpha_i  \in \Sigma_2$, there is $k_{i} \in \Z^d$ such that 
 \[ \alpha_i - \bar \alpha = 2i\pi \langle k_{i}, \omega \rangle .\]
\end{enumerate}

Then we construct $W \in \CC^\infty(\mathbb{T}^d,Gl(n, \C)), W = \text{diag}(w_k)_{k=1, \cdots, l}$
(where $w_k$ is a submatrix associated to the generalized eigenspaces of the eigenvalue $\alpha_k$)
such that for all $\theta\in \mathbb{T}^d$ and all $i=1,\dots, n$,

\begin{itemize}
\item \[ w_i(\theta) = e^{2i\pi \langle k_{i}, \theta \rangle}I\] if $[\alpha_i]$ is in case 1;
\item 
\[w_i(\theta) = e^{4i\pi \langle k_{i}, \theta \rangle}I\] 

if $[\alpha_i]$ is in case 2.
\end{itemize}
The relation 

\begin{equation}\label{conjugaison-W}\partial_{\frac{\omega}{2}}W =BW-WB' 
\end{equation}

\noindent
defines a matrix $B'$ of dimension $n \times n$ (notice that $B'-B$ is a diagonal matrix, since the coefficients for $B$ and $B'$ outside the diagonal are the same),  
whose diagonal coefficients are either real (in case 1) or come by pairs of complex conjugates with the same multiplicity (in case 2). 

\bigskip
Now we will prove that, in case 2, the two blocks of $B'$ corresponding to complex conjugate eigenvalues are algebraically conjugate, which will imply that they have the same Jordan structure. 

The relation \eqref{conjugaison-B-blocks} combined with \eqref{conjugaison-W} implies that for all $i,j$, if we denote $\tilde U_i^j(\theta ) = U_i^j(2\theta)$,

$$\partial_{\frac{\omega}{2}} (w_i^{-1}\tilde U_i^j\bar{w}_j)
= B_i' (w_i^{-1}\tilde U_i^j\bar{w}_j) - (w_i^{-1}\tilde U_i^j\bar{w}_j) \bar{B}_j'
$$

\noindent and by construction and the second statement of Lemma \ref{lemma1}, $w_i^{-1}\tilde U_i^j\bar{w}_j$ is constant. 
Up to a permutation, one can assume that the blocks of $B$ corresponding to the eigenvalues in the same equivalence class $\mathcal{E}$ are next to each other and can be grouped in a block $B_{\mathcal{E}}$.
For any $\alpha\in\sigma(B)$, letting $W_{[\alpha]}=\text{diag}(w_i,\alpha_i\in[\alpha])$ and $\tilde U_{[\alpha]}$ be the block of $\tilde U(\theta):=U(2\theta)$ corresponding to $[\alpha]$ (so $\tilde U=\text{diag}(\tilde U_{[\alpha]})$), then $W_{[\alpha]}^{-1}\tilde U_{[\alpha]} \bar W_{[\alpha]}$ satisfies 

$$\partial_{\frac{\omega}{2}} (W_{[\alpha]}^{-1}\tilde U_{[\alpha]} \bar W_{[\alpha]})
=B'_{[\alpha]}W_{[\alpha]}^{-1}\tilde U_{[\alpha]} \bar W_{[\alpha]}-W_{[\alpha]}^{-1}\tilde U_{[\alpha]} \bar W_{[\alpha]}\bar B'_{[\alpha]}.
$$

\noindent Moreover, $W_{[\alpha]}^{-1}\tilde U_{[\alpha]} \bar W_{[\alpha]} $ is constant and also invertible 
since $\tilde U$ is invertible. Thus $B'_{[\alpha]}$ and $\bar B'_{[\alpha]}$ are algebraically conjugate, which means that two blocks of $B' $ corresponding to complex conjugate eigenvalues have the same Jordan structure.

\bigskip

The matrix $B'$ is not necessary real, but can be conjugated to a real matrix applying lemma \ref{spectre-conjugué} in appendix.
\end{proof}

\bigskip 
We now prove the first main theorem:
\begin{theorem}[Real reducibility]
Let $A: \T^d \rightarrow gl(n, \R)$ such that $X_{\omega, A}$ is reducible. Then $X_{\frac{\omega}{2},A_2}$ is real reducible, where 
\[ A_2(\theta) := A(2\theta), \quad \forall \theta \in \T^d.\]
\end{theorem}
\begin{proof}
By assumption, there exist $B \in gl(n, \C)$ and $Z : \T^d \rightarrow Gl(n, \C)$ such that
\[\partial_\omega Z = AZ - ZB \]
and then 
\[ \partial_\omega \bar Z = A \bar Z - \bar Z \bar B\]
and
\[ \partial_\omega Z^{-1} = B Z^{-1} - Z^{-1}A.\]
Let $P$ an invertible matrix such that $B=PJP^{-1}$, where $J$ is in normal Jordan form. Let
\[ U = P^{-1}Z^{-1}\bar Z \bar P, \]
then
\begin{align*}
\partial_\omega U & =  \partial_\omega(P^{-1}Z^{-1}\bar Z\bar P) \\
& = P^{-1}(\partial_\omega (Z^{-1})\bar Z + Z^{-1}\partial_\omega(\bar Z)) \bar P  \\
& = P^{-1}((B Z^{-1} - Z^{-1}A)\bar Z + Z^{-1}(A \bar Z - \bar Z \bar B))\bar P \\
& = JU - U \bar J.
\end{align*}
We can apply proposition \ref{reducprincipal} and deduce that there exist $W : \T^d \rightarrow Gl(n, \C)$ and $J' \in gl(n, \R)$ such that
\[ \partial_{\frac{\omega}{2}}W = JW - WJ'.\]
Let $Z'(\theta) = Z(2\theta)PW(\theta)$. Then

\begin{align*}
\partial_{\frac{\omega}{2}}Z'(\theta)  & = \partial_{\frac{\omega}{2}}Z(2\theta)PW(\theta) + Z(2\theta) P\partial_{\frac{\omega}{2}}W(\theta) \\
						       & = (A(2\theta)Z(2\theta) - Z(2\theta)B)PW(\theta) + Z(2\theta)P(JW(\theta) - W(\theta)J') \\
						       & = A(2\theta)Z(2\theta)PW(\theta) - Z(2\theta)PW(\theta)J' \quad \text{ since $BP = JP$} \\
						       & = A_2(\theta)Z'(\theta)-Z'(\theta)J'.
\end{align*}
Notice that $Z'$ is not necessarily real, however by proposition \ref{redpreli}, there exists $\lambda \in \R$ such that \[ \Re Z' + \lambda \Im Z' : \T^d \rightarrow Gl(n, \R) \] conjugates the two real cocycles $X_{\frac{\omega}{2}, A_2}$ and $X_{\frac{\omega}{2}, J'}$. 

\end{proof}

In the remainder of the article, we will prove the second main result, which is that almost reducibility for a real cocycle implies real almost reducibility.

\section{Construction of a real change of variables}\label{section3}

\subsection{Lemmas}

In this section, we prove a few lemmas about the trace of a system and the determinant of a conjugation. They will be used to construct real changes of variables for real almost-reducible cocycles. Here, we need an arithmetical condition on $\omega$.

\subsubsection{A small divisor lemma}

Let $f:\T^d\to\C$ be of class $\CC^\infty$ and consider the equation
\begin{equation}\label{eq} 
\begin{cases}
\partial_\omega g = f-\hat f (0)\\
\hat g(0)=0.
\end{cases} 
\end{equation}

\medskip

\begin{lemma}\label{l3.1}
If $\omega \in \mathcal{DC}(\kappa, \tau)$, then there exists a unique solution $g:\T^d\to\C$
to \eqref{eq} and it satisfies
$$\Vert g \Vert_{\CC^r}\leq C_{r,d}  \frac{1}{\kappa}\Vert f\Vert_{\CC^{\tau+d+1}},\quad \forall r\geq 
0,$$
where $C_{r,d}$ is a constant depending only on $r,d$.
\end{lemma}

\begin{proof}
Developing $g$ in Fourier series, we get
$$g(\theta) \simeq \sum_{k\neq 0} \frac{\hat f(k)}{2i\pi \langle k,\omega\rangle}e^{2i\pi \langle k,\theta\rangle} .$$
Then for all $s \geq 1$, since $\omega\in\mathcal{DC}(\kappa, \tau)$,
\begin{align*}
    \Vert g \Vert_{\CC^0}& \leq \sum_{k\neq 0} \frac{\vert k \vert^{\tau} \vert \hat f(k) \vert}{2\pi \kappa} \\
    & \leq C_s   \frac{1}{\kappa}
\sum_{k\neq 0}\vert k \vert^{\tau -s}\sup_{\vert \alpha \vert=s} \vert \widehat{\partial^\alpha f(k)} \vert \\
& \le C_s 
 \frac{1}{\kappa} \Vert f\Vert_{\CC^{s}} \sum_{k\neq 0} \vert k \vert^{\tau-s} \\
 & \leq C_{s,d} 
 \frac{1}{\kappa} \Vert f\Vert_{\CC^{s}}  \sum_{j> 0}j^{\tau-s+d-1} 
\end{align*}
and this converges if $s\geq \tau+d+1$.

The higher derivatives are obtained by differentiating in the Fourier series.
\end{proof}

\subsubsection{Trace and determinant}

\begin{lemma}\label{trace-constante}
Let $A: \T^d\to gl(n,\C)$ and $\omega \in \mathcal{DC}(\kappa, \tau)$. There exist $Z: \T^d\to Gl(n,\C)$ and $B\in  gl(n,\C)$ of constant trace and of class $\CC^\infty$ such that
$$\partial_\omega Z=AZ-ZB.$$
\end{lemma}

\begin{proof}
Let $B=A-(\Tr A -\int_{\T^d}\Tr A (\theta)d\theta)I$. Then $B$ has constant trace.  If $f = \Tr A -\int_{\T^d}\Tr A (\theta)d\theta$, then $\hat{f}(0)=0$ and the equation
\[\left\{\begin{array}{l}
\partial_\omega g=f \\
\hat g(0)=0
\end{array}\right.\]
has a unique solution $\T^d\to\C$ (which is of class $\CC^\infty)$ by Lemma \ref{l3.1}. Let now $Z=e^g I: \T^d\to Gl(n,\C)$, then $Z$ satisfies 
\[ \partial_\omega Z = AZ-ZB.\]
\end{proof}

\begin{lemma}\label{trace}Let $A,B,F,Z: \mathbb{T}^d\to gl(n,\C)$, with $Z$ differentiable. If 
\[ \partial_\omega Z = AZ-ZB +F \]
then 
\[ \partial_\omega \det Z=\Tr(A-B) \det Z +\Tr(FZ^{adj}) \]
where $Z^{adj}$ is the transpose of the cofactor matrix of $Z$.

\end{lemma}

\begin{proof}  We have, for all $\theta \in \T^d$, $Z(\theta)Z^{adj}(\theta)=\det Z(\theta)\cdot I$, so
$$\partial_\omega \det Z(\theta) \cdot I = \partial_\omega Z(\theta)Z^{adj}(\theta)+Z(\theta)\partial_\omega Z^{adj}(\theta),$$
\noindent and taking the trace we find
$$n\partial_\omega \det Z  = \Tr((AZ-ZB+F)Z^{adj}) + \Tr(Z\partial_\omega Z^{adj})=$$
$$=\Tr(A-B)\det Z + \Tr(FZ^{adj})+ \Tr(Z\partial_\omega Z^{adj}).$$

We want to show that $(n-1)\partial_\omega \det Z = \Tr (Z\partial_\omega Z^{adj})$. From the formula of the differential of the determinant:
\[D(\det A)[ H] = \Tr (
A^{adj}H)  \]
where $A^{adj}$ is the transpose of the cofactor matrix of $A$,
and from the formula of the derivative of composite functions
\[ \frac{\partial}{\partial t} (F \circ f)_{\vert _{t=0}} = D_{f(t)}F[\frac{\partial f}{\partial t}]_{\vert_{t=0}}\]
where here
\[ F = \det, f = Z(\theta + t\omega) \Rightarrow f(0) = Z(\theta), \quad \frac{\partial f}{\partial t}_{\vert_{t=0}} = \partial_\omega Z\]
and therefore
\[ D _{f(t)} F _{\vert_{t=0}}
= \Tr (
Z^{adj}\, 
). \]
Hence,
\begin{align*}
\partial_\omega \det Z & = Tr(Z^{adj}\partial_{\omega}Z)   = Tr(\partial_{\omega} Z Z^{adj}) \\
				&   = Tr(\partial_\omega (Z Z^{adj})) - Tr(Z \partial_\omega Z^{adj}) \\
				&   = Tr(\partial_\omega (\det Z Id)) -  Tr(Z \partial_\omega Z^{adj}) \\
				&   = n \partial_\omega \det Z - Tr(Z \partial_\omega Z^{adj}).
\end{align*}
Finally, 
\begin{align*}
    \partial_\omega \det Z &=n \partial_\omega \det Z - Tr(Z \partial_\omega Z^{adj}) \\
    &=\Tr(A-B)\det Z + \Tr(FZ^{adj}).
\end{align*}

\end{proof}

\medskip

\begin{remark}
This lemma does not require any arithmetical condition on $\omega$.   
\end{remark}

\medskip

The following lemma will be used to construct a real invertible conjugation out of a complex one.

\begin{lemma}\label{det-lambda}
 If $Z:\T^d\to  Gl(n, \C)$  satisfies
$$\left \vert \int_{\T^d}  \det {Z}(\theta) d\theta\right\vert =1,$$
then there exists $\lambda \in [-1,1]$ such that 
\begin{equation}\label{moduleint}\left \vert\int_{\T^d}  \det (\Re Z(\theta)+\lambda \Im Z(\theta)) d\theta\right\vert \geq C_n>0,\end{equation}
and the constant $C_n$ only depends on $n$.
\end{lemma}

\begin{proof}
Let
$$P(\lambda) = \int_{\T^d}  \det (\Re Z(\theta)+\lambda \Im Z(\theta)) d\theta.$$ 
Then $P$ is a polynomial of degree $n$ which is non zero because $ P(i)$ is a complex number of unit modulus by assumption.
Hence
$$P(\lambda) = \mu (\lambda - \alpha_1)\dots (\lambda-\alpha_n)$$ 
for some $\mu \neq 0$, $\alpha, \cdots, \alpha_n \in \C$, and we have
$$\mu = \frac{P(i) }{(i - \alpha_1)\dots (i-\alpha_n)}.$$

\noindent By the Pigeon hole principle, there exists $k\in\{0,\dots, n\}$ such that 
$$(-1+\frac{2k}{n+1},-1+\frac{2(k+1)}{n+1})\ \bigcap\  \{\Re \alpha_1,\dots, \Re \alpha_n\}=\emptyset.$$ 
If $\lambda_0 = -1+\frac{2k+1}{n+1}$, then one has
$$ \vert \lambda_0 - \alpha_j \vert \geq \vert \lambda_0-\Re \alpha_j \vert \geq \frac{1}{n+1},\quad \forall j. $$

\noindent 
If $|\alpha_j|\leq 3$, then $\frac{|\lambda_0-\alpha_j|}{|i-\alpha_j|}\geq \frac{1}{4(n+1)}$. 
If now $\vert \alpha_j \vert >3$, we have $|\frac{\lambda_0}{\alpha_j}|\leq \frac{2}{3}$ therefore $|\frac{\lambda_0}{\alpha_j}-1|\geq\frac{1}{3}$, and also $|\frac{i}{\alpha_j}|\leq \frac{1}{3}$ which implies $|\frac{i}{\alpha_j}-1|\leq \frac{4}{3}$. Thus

\[ \frac{\vert \lambda_0-\alpha_j \vert }{ \vert i-\alpha_j \vert }=\frac{ \vert \frac{\lambda_0}{\alpha_j}-1 \vert }{ \vert \frac{i}{\alpha_j}-1 \vert}\geq\frac{1}{4} \]
so
$$\vert P(\lambda_0) \vert \geq  \prod_{\vert \alpha_j \vert \leq 3} \frac{\vert \lambda_0-\alpha_j \vert }{\vert i-\alpha_j \vert}  \prod_{\vert \alpha_j \vert > 3}\frac{\vert \lambda_0-\alpha_j \vert }{\vert i-\alpha_j \vert }
\geq \frac{1}{(4(n+1))^n},$$
so $\lambda_0$ satisfies $\eqref{moduleint}$.

\end{proof}

\subsection{Construction of a sequence of real changes of variables}

\begin{proposition}\label{presqueconjugaisonreelle}
Let $X_{\omega, A}$ be a real cocycle which is almost reducible to real matrices, i.e.  there exist sequences  $Z_j:\T^d\to Gl(n,\C)$, $B_j\in \mathcal{M}(n,\R)$ and $F_j : \T^d \rightarrow \mathcal{M}(n, \C)$ such that
$$\partial_\omega Z_j=AZ_j-Z_j(B_j+F_j)$$
and
$$ \lim_{j\rightarrow +\infty} \Vert Z_j^{\pm 1} \Vert_{\CC^r}^m \Vert F_j \Vert_{\CC^r}=0, \quad \forall m,r\geq 0.$$

\noindent If $\omega\in \mathcal{DC}(\kappa,\tau)$,
then  $X_{\omega, A}$ is real almost reducible.

\end{proposition}

\medskip

\begin{proof} 
By Lemma \ref{trace-constante} we can assume  $\Tr A$ is constant, and since $\Tr( A)I$ commutes with $Z_j$,

$$\partial_\omega Z_j = (A-\frac{1}{n}Tr(A)\cdot I)Z_j -Z_j(B_j +F_j-\frac{1}{n}Tr(A)\cdot I)
$$

\noindent therefore, by replacing $B_j$ with $B_j-\frac{1}{n}Tr(A)\cdot I$, we can assume that 
$$\Tr A=0.$$

\medskip

\noindent We have, for all $\theta\in\mathbb{T}^d$,
$$\Vert \det Z_j^{\pm 1}(\theta)\Vert\le C_n\Vert Z_j^{\pm 1}\Vert_{\CC^0}^n  $$
and, hence for all $\theta\in\mathbb{T}^d$,
$$\Vert \det Z_j^{\pm 1}(\theta)\Vert=\Vert \frac1{\det Z_j^{\mp 1}(\theta)}\Vert\ge \frac1{C_n}\Vert Z_j^{\mp 1}\Vert_{\CC^0}^{-n}.$$
So the quantity
$$a_j=\big(\int_{\T^d} \det Z_j(\theta)\, d\theta\big)^{\frac 1n}\in\mathbb{C}$$
satisfies
\begin{equation}\label{inequality36}0 < \frac1{ C_n' \Vert Z_j^{-1} \Vert_{\CC^0} }\le  \vert a_j\vert\le C_n' \Vert Z_j\Vert_{\CC^0}.\end{equation}

\noindent Define now
$$\tilde Z_j=\frac1{a_j} Z_j.$$
Then 
\[\partial_\omega \tilde{Z} _j=A\tilde{Z}_j - \tilde{Z}_j(B_j+F_j)\]
and for all $r\in \N$, \eqref{inequality36} implies
$$
\Vert \tilde{Z}_j\Vert _{\CC^r}\leq C_n'\Vert Z_j^{-1}\Vert_{\CC^0} \Vert Z_j\Vert_{\CC^r}, \quad \Vert \tilde{Z}_j^{-1}\Vert _{\CC^r}\leq C_n'\Vert Z_j\Vert_{\CC^0} \Vert Z_j^{-1}\Vert_{\CC^r},$$

\noindent which implies that
$$ \lim_{j\rightarrow +\infty} \Vert \tilde Z_j^{\pm 1} \Vert_{\CC^r}^m \Vert F_j \Vert_{\CC^r}=0, \quad \forall m,r\geq 0.$$
Replacing $Z_j$ by $\tilde Z_j$, we can therefore simply assume that
\begin{equation}\label{det=1}\int_{\T^d} \det Z_j(\theta) \, d\theta=1,\quad \forall j.\end{equation}

\medskip

\noindent Applying Lemma \ref{trace} with $B=B_j$ and $F=-Z_jF_j$, we have 
\begin{align*}
\partial_\omega \det {Z}_j &= \Tr(A-B_j)\det {Z}_j -\Tr(Z_jF_j Z_j^{adj}) \\
& =-\Tr(B_j+F_j)\det Z_j.
\end{align*}
Hence
$$0=\int_{\T^d} \partial_\omega \det {Z}_j(\theta) d\theta=-\Tr(B_j)-\int_{\T^d}\Tr(F_j(\theta))  \det {Z}_j(\theta) d\theta.$$
Replacing $B_j$ by $B_j + \frac{1}{n}\int_{\T^d} \Tr(F_j(\theta))\det Z_j(\theta)d\theta\cdot  I$ which is real since $B_j$ is real and from the previous equality, and replacing $F_j$ by $F_j - \frac{1}{n}\int_{\T^d}\Tr(F_j(\theta)) \det Z_j(\theta)d\theta\cdot I$ (which still satisfies condition convergence, since $F_j$ does), we can therefore simply assume that
$$\Tr(B_j)=0,\quad \forall j.$$
\medskip

\noindent 
By Lemma  \ref{det-lambda}, which can be applied thanks to \eqref{det=1}, there exists a constant $C_n >0$, and for all $ j$ there exist a $\lambda_j\in [-1,1]$, 
such that 
$$W_j= \Re Z_j + \lambda_j \Im Z_j$$
satisfies
\begin{equation}\vert \widehat{\det W_j}(0)\vert=\vert \int_{\T^d} \det W_j(\theta) d \theta \vert \ge C_n,\quad .\label{W-chapeau}\end{equation}

\noindent 
Clearly, since $A$ and ${B}_j$ are real,
$$\partial_\omega W_ j = AW_j -W_j (B_j+ G_j)$$
where $G_j =W_j^{-1}( \Re({Z}_j{F}_j) +\lambda_j \Im({Z}_j{F}_j))$. In particular
$$ \lim_{j\rightarrow +\infty} \Vert W_j\Vert_{\CC^r}^m \Vert G_j \Vert_{\CC^r}=0, \quad \forall m,r\geq 0.$$

\medskip

\noindent There remains to study the inverse of $W_j$.

\noindent By Lemma \ref{trace}, we have 
$$\partial_\omega \det {W}_j = \Tr(A-B_j)\det {W}_j -\Tr(W_jG_j W_j^{adj})=$$
$$=-\Tr(W_jG_j W_j^{adj})=-Tr(G_j\text{det}(W_j)):=H_j.$$
Since
$$\Vert H_j\Vert_{\CC^r}\le C_r \Vert G_j\Vert_{\CC^r}\Vert W_j\Vert_{\CC^r}^n$$
we have
$$\lim_{j\to + \infty}\Vert H_j\Vert_{\CC^r}= 0,\quad \forall r\geq 0.$$

\noindent 
From 
$$ \partial_\omega \det W_j = \partial_{\omega} (\det W_j + \widehat{\det W_j}(0))= H_j - \hat H_j(0) $$
and applying Lemma \ref{l3.1}, if $\omega \in \mathcal{DC}(\kappa,\tau)$, we then have for all $r \geq 0$,
$$\Vert \det W_j- \widehat{\det W_j}(0)\Vert_{\CC^r}\le  C_{r,d}\frac{1}{\kappa } \Vert H_j \Vert_{\CC^{\tau+d+1}}.$$

\noindent
So for $j$ sufficiently large and thanks to \eqref{W-chapeau}, we have 
$$\vert \det W_j(\theta)\vert\ge \frac12 C_n,\quad \forall \theta\in \T^d.$$
This implies that $W_j$ is invertible, and for $j$ sufficiently large, using the well known formula $W_j^{-1} = \frac{1}{\det W_j}W_j^{adj}$ (where $W_j^{adj}$ is the transpose of the cofactor matrix of $W_j$),
$$\Vert W_j^{-1} \Vert_{\CC^0}\le C_n'\Vert W_j\Vert_{\CC^0}^{n-1}$$
and, by \eqref{2},
$$
\Vert W_j^{-1}\Vert_{\CC^r} \leq C_r \Vert W_j^{-1} \Vert_{\CC^0}^{r+1}\Vert W_j \Vert_{\CC^r}^r,\quad \forall r.$$

\noindent In particular
$$ \lim_{j\rightarrow +\infty} \Vert W_j^{-1}\Vert_{\CC^r}^m \Vert G_j \Vert_{\CC^r}=0, \quad \forall m,r\geq 0.$$
\end{proof}

\section{Jordan normal form with estimates on the conjugation matrix}\label{section4}

This section is devoted to conjugating a matrix to its Jordan normal form with sufficient estimates.

\subsection{Column echelon form by an algebraic conjugation}

\begin{definition}
    We say that a matrix $A$ is in \textit{column echelon form} if it has strictly increasing column lengths, except the first columns which can be zero.
    Its \textit{pivots} are the last non zero coefficient of each non zero column.
\end{definition}

We will now conjugate a nonzero nilpotent matrix $A \in gl(n, \C)$ to a nilpotent one in column echelon form. Moreover, the conjugation will be unitary. \\
Let $m$ the index of $A$, that is to say the smallest integer $m$ such that $A^m =0$ (here $m\geq 2$ since $A$ is nonzero). Denote $L : \C^n \rightarrow \C^n$ the linear map represented by $A$ in the canonical basis, and $K_j = \ker L^j$ the kernel of the iterates of $L$, for all $j \in \{0, \dots, m \}$. Then we have
\[ \{ 0 \} = K_0 \subset_{\neq} K_1 \subseteq K_2 \subseteq \cdots \subseteq K_{m-1} \subset_{\neq} K_m = \C^n.\]
Let for all $j \in \{1, \dots, m\}$, $U_j = K_j \cap K_{j-1}^{\perp}$ where $K_{j-1}^{\perp}$ is the orthogonal complement of $K_{j-1}$ in $\C^n$ equipped with the standard inner product. This implies the orthogonal direct sum
\[ K_j = K_{j-1} \oplus^{\perp} U_j\]
from which we can define $\proj_{U_j}: K_j \rightarrow U_j$ the orthogonal projection onto $U_j$.
Finally,
\[ K_j = U_1 \oplus^{\perp} U_2 \oplus^{\perp} \cdots \oplus^{\perp} U_j.\]
We will also denote
\[ r_j = \dim U_j.\]
\begin{lemma} \label{lemmaech}
With the above notations, 
\begin{enumerate}[(i)]
    \item $L(K_j)\subseteq K_{j-1}$ for all $j \in \{1, \dots, m\}$,
    \item the restriction $L \vert_{U_j}$ of $L$ to $U_j$ is injective for all $j \in \{2, \dots, m\}$,
    \item $L(U_j) \cap K_{j-2} = \{0 \}$ for all $j \in \{2, \dots, m\}$,
    \item $(\dim U_i)_{i=1, \dots,m}$ is a non increasing sequence.
\end{enumerate}
\end{lemma}

\begin{proof}
    \begin{enumerate}[(i)]
    \item For all $j\in \{1, \dots, m\}$, 
    \[u \in K_j \Leftrightarrow L^ju = 0 \Leftrightarrow L^{j-1}(Lu) = 0 \Leftrightarrow Lu \in K_{j-1}.\]
    \item Let $j \in \{2, \dots, m\}$, and let $u \in U_j = K_j \cap K_{j-1}^{\perp}$ with $u\neq 0$. Then $u \notin K_{j-1}$ and then $L^{j-1}u \neq 0$, which implies (since $j-1 \geq 1$) $Lu \neq 0$.
    \item Let $j \in \{2, \dots, m\}$ and let a nonzero $u \in U_j = K_j \cap K_{j-1}^\perp$, then $u \notin K_{j-1}$, and from $(i)$, $Lu \in K_{j-1}$. Reasoning by absurdity, suppose that $Lu \in K_{j-2}$, then $L^{j-1}u = L^{j-2}(Lu) = 0$, which implies $u \in K_{j-1}$ leading to a contradiction.
    \item  Let $j \in \{ 2, \dots, m\}$ and $u_1, \dots, u_r$ linearly independent vectors in $U_j \subset K_j$. From $(ii)$, $Lu_1, \dots, Lu_r$ are linearly independent vectors in $K_{j-1} = K_{j-2} \oplus^\perp U_{j-1}$ and from $(iii)$, $\proj_{U_{j-1}}(Lu_1), \dots, \proj_{U_{j-1}}(Lu_r)$ are linearly independent vectors in $U_{j-1}$, which implies that $\dim U_{j-1} \geq \dim U_j$.
    \end{enumerate}
\end{proof}
\begin{proposition}\label{echok}
    Let $A$ a nilpotent matrix of index $m\geq 2$. With the above notations, for all $j\in \{1, \dots, m\}$, there exists an orthonormal basis $\BB^j = \{u_1^j, \dots, u_{r_j}^j \}$ of $U_j$ (where $r_j=\dim U_j$) such that, for all $j\geq 2$ and for all $k \in \{1, \dots, r_j\}$, letting $K_{-1} = \{0\}$,
\[Lu_k^j \in \Span (u_1^{j-1}, \dots, u_k^{j-1}) \oplus^\perp K_{j-2}\]
and
\[\langle Lu_k^j,u_k^{j-1} \rangle \neq 0.\]
\end{proposition}

\begin{proof}
  Let $\BB^m= \{u_1^m, \dots, u_{r_m}^m\}$ an orthonormal basis of $U_m$. From $(ii)$ of lemma \ref{lemmaech}, $\{Lu_1^m, \dots, Lu_{r_m}^m\}$ are linearly independant vectors of $K_{m-1} = K_{m-2} \oplus^\perp U_{m-1}$ and from $(iii)$ of lemma \ref{lemmaech}, $\{v_1 = \proj_{U_{m-1}}(Lu_1^m), \dots, v_{r_m} = \proj_{U_{m-1}}(Lu_{r_m}^m) \}$ are linearly independant vectors of $U_{m-1}$. Apply Gram–Schmidt on $\{v_1, \dots, v_{r_m} \}$.
   Then there exists an orthonormal basis $\BB = \{u_1', \dots, u_{r_{m-1}}'\}$ (taking $\{ v_1, \dots, v_{r_m} \}$ and completing into an orthonormal basis if $r_{m-1}-r_{m} >0$) of $U_{m-1}$ such that for all $k \in \{1, \dots, r_m \}$ (recall $r_{m} \leq r_{m-1}$),
    \[v_k \in \Span (u_1', \dots, u_k') \]
and
\[\langle u_k', v_k \rangle \neq 0.\]
Then for all $k \in \{ 1, \cdots, r_{m} \}$,
    \[Lu_k^m \in \Span (u_1', \dots, u_k') \oplus^{\perp}K_{m-2} \]
and
\[\langle Lu_k^m,u_k' \rangle = \langle u_k', v_k \rangle \neq 0,\]
and let $\BB^{m-1}=\BB$. Now if $m=2$ we are done, and if $m\geq 3$, we construct every $\BB^{j}$ the same way from $U_{j+1}$ and $\BB^{j+1}$.
\end{proof}

\begin{corollary}\label{coro-ech}
    Let $L$ be
    a nilpotent linear application of index $m\geq 2$. With the above notations, there exists an orthonormal basis $\BB = \{u_1, \dots, u_n \}$ of $
    \C^n$ such that $L$ is represented in this basis by the matrix
    \begin{equation} \label{formedeA}A=\left(\begin{array}{ccccc}
    0 & A_1^2 & \ast  & \cdots & \cdots \\
    0 & 0 & A_2^3 & \ddots  & \vdots \\
    0 & 0 & 0 & \ddots & \ast \\
    0 & 0 & 0 & 0 & A_{m-1}^m \\
    0 & 0 & 0 & 0  & 0
    \end{array}\right)\end{equation}
    where, for all $j \in \{2, m\}$, $A_{j-1}^j$ is a $r_{j-1}\times r_j$ matrix of the form
\[
\left(\begin{array}{ccc}
\alpha_1 & \ast & \ast  \\
0 & \ddots & \ast  \\
0 & 0 & \alpha_{r_j}  \\
0 & 0 & 0  \\
\vdots & \vdots & \vdots  
\end{array}\right)
\]
if $r_{j-1} > r_j$
or
\[
\left(\begin{array}{ccc}
\alpha_1 & \ast & \ast  \\
0 & \ddots & \ast  \\
0 & 0 & \alpha_{r_j}  
\end{array}\right)
\]
if $r_{j-1} = r_j$, and with $\alpha_i \neq 0$ for all $i\in \{1, \dots, r_j \}$.
\end{corollary}

\begin{proof}
Concatenating the basis $\BB^i$ obtained for all $i \in \{1, \dots, m \}$ in proposition \ref{echok}, blocks $A_{j-1}^j$ have coefficients
\[  \langle u_k^{j-1}, Lu_{l}^j \rangle \]
for all $k \in \{ 1, \dots, r_{j-1} \}$, $l\in \{1, \dots, r_j\}$.
\end{proof}

The following Lemma will be useful in analyzing the behaviour of the kernel of a matrix of form \eqref{formedeA} when non-zero coefficients are removed.

\begin{lemma}\label{lKerSpectrum}
Let $A$ of the form \eqref{formedeA}.
Let $B$ with the same block-triangular form as $A$, that is to say, according to the notation of corollary \ref{coro-ech}, for all $i\geq j$,
\[ A_i^j = 0 \Rightarrow B_i^j = 0.\]
Denote by $L_A$ and $L_B$ the linear maps represented by $A$ and $B$ respectively in the canonical basis. Then

\begin{itemize}
\item[$(i)$]
$$ \ker L_B^j\supseteq \ker L_A^j,\quad \forall j.$$

\item[$(ii)$]
If 
$$ \ker L_B^j= \ker L_A^j,\quad \forall j,$$
\noindent then the block 
$B^j_{j-1}$ is of maximal rank $r_j$ for all $j$.
\end{itemize}

\end{lemma}

\begin{proof} $(i)$ follows from the assumption on $B$.

To see $(ii)$, notice that
$$K^A_j = \ker L_A^j=\R^{r_1+\dots+r_j}\times \{0\}$$
and
$$U^A_j = \{0\}\times \R^{r_j}\times \{0\}.$$
Then
$$L_A:U^A_j\ni  \begin{pmatrix} 0\\ u\\ 0\end{pmatrix}\mapsto  \begin{pmatrix} 0\\ A^j_{j-1} u\\ 0\end{pmatrix}\in U^A_{j-1}$$
which is an injective map.

If now $ \ker L_B^j= \ker L_A^j$
then, by the second statement of Lemma \ref{lemmaech},
$$L_B: U_j^A\ni \begin{pmatrix} 0\\ u\\ 0\end{pmatrix}\mapsto  \begin{pmatrix} 0\\ B^j_{j-1} u\\ 0\end{pmatrix}\in U^A_{j-1}$$
is injective. This is the same as $B^j_{j-1}$ being of maximal rank $r_j$.

\end{proof}

\subsection{Reduced column echelon form}

\begin{definition}
A matrix  $A \in gl(n, \C)$ is on \textit{reduced column echelon form}  if and only if it is on column echelon form with all pivots equal to $1$.
\end{definition}

We shall  conjugate a (nonzero) nilpotent matrix on column echelon form to reduced column echelon form modulo a perturbation
with control on the conjugation.

\begin{lemma}\label{lScaling}
Let $B$ be a matrix on column echelon form with pivots $(\alpha_j)_j$ and let $\delta > 0$.
If 
$$\vert\alpha_j\vert \ge \delta>0,\quad \forall j,$$
then there exists a diagonal matrix $S$,
$$\Vert S^{\pm}\Vert\le C(\frac{\Vert B\Vert}{\delta})^{\frac n2},$$
such that $S^{-1}BS$
is on reduced column echelon form.

\noindent The constant $C$ only depends on $n$.
\end{lemma}

\begin{proof} Let $B$ in column echelon form, then up to a change of orthonormal basis, it has the form of Corollary \ref{coro-ech}; we will look for a block-diagonal $S$
where the diagonal blocks $S_j$ are diagonal and their dimensions are $r_j$ (the number of columns in $B_{j-1}^j$).
Then for all $j$, we want to find $S_j$ such that 
$$S_{j-1}^{-1}B_{j-1}^j S_j=
\left(\begin{array}{ccc}
1& \ast & \ast  \\
0 & \ddots & \ast  \\
0 & 0 & 1\\
\vdots & \vdots & \vdots  
\end{array}\right)$$

\noindent where $S_{j-1}$ is given. These equations can be solved uniquely, one by one, starting with $S_1=I$. 
For a given $j$, it is easy to notice that all coefficients of $S_j$ are $\geq \delta^n$ (as it is either $1$, or a product of the inverse of some pivots of $B$). In the same way, each coefficient of $S_j^{-1}$ is $\leq \Vert B \Vert^n$. We thus obtain
\[\Vert S \Vert \leq \frac{1}{\delta^n}, \quad \Vert S^{-1} \Vert \leq \Vert B \Vert^n. \]
One can then just multiply $S$ by $\sqrt{\Vert B \Vert^n\delta^n}$ to get the estimate.

 \end{proof}



\begin{proposition}\label{echelonnee}
    Let $A$ a nilpotent 
    block diagonal matrix. 
     Let $\varepsilon >0$ and $(\delta_k)_{k\ge 0}$ a positive decreasing sequence such that
     \begin{equation}\label{decroissantedelta}\varepsilon^{\delta_{1}}+\dots+\varepsilon^{\delta_{k}}\le 2\varepsilon^{\delta_{k}},\quad \forall k\ge 1.\end{equation}
     
\noindent    Then there exists   $1\le k\le \frac{n^2}{2}$ and $S\in Gl(n,\C)$ with
      \begin{equation} \label{estimateofS} \Vert S^{\pm}\Vert \leq  C(\Vert A\Vert+2\varepsilon^{\delta_{k-1}})^{\frac n2}  \varepsilon^{-\frac n2 \delta_{k}}  \end{equation}
  such that  
\[ S^{-1}A S =  A'+  F,\]
with 
\begin{equation}\label{estimateofF}\Vert  F \Vert \leq C(\Vert A\Vert+2\varepsilon^{\delta_{k-1}})^{n} \varepsilon^{- n \delta_{k}+\delta_{k-1} },\end{equation}
where the constant $C$ only depends on $n$, and $S$ and $A'$ are 
block diagonal with the same block decomposition as $A$, and each 
block of $A'$ is on reduced column echelon form.
    
 \end{proposition}

\begin{proof}
For any nilpotent matrix $A$, define
$$\sigma\\ker(A)=(\dim \ker L_A,\dim \ker L_A^2,\dots,\dim \ker L_A^{n-1})$$
where $L_A$ is associated to $A$ in the canonical basis.
This is a non decreasing sequence of integers $\in \llbracket 1,n \rrbracket$ and
\begin{equation}
    \sigma\ker(A)=(n,n,\dots, n)\quad \iff\quad A=0. \label{algostop}\end{equation}
If $B$ is another nilpotent matrix we say that $\sigma\ker(A)> \sigma\ker(B)$ if and only if
$$\dim \ker L_A^j\quad \begin{cases} \ge  \dim \ker L_B^j & \textrm{\ for\ all\ } j\\  >  \dim \ker L_B^j & \textrm{\ for\ some \ } j\end{cases}$$
(this is of course not a total ordering).

\medskip

By applying Corollary \ref{coro-ech} to each block of $A$ we can assume, without restriction, that each block of $A_1=A$ is on column echelon form. By induction:

\textbf{Base case:} If no pivot is $\leq \varepsilon^{\delta_1}$, applying Lemma \ref{lScaling} to each block of $A_1$, there exists a diagonal matrix $S_1$ with $ \Vert S_1^{\pm 1} \Vert \leq C (\frac{\Vert A_1 \Vert}{\varepsilon^{\delta_1}})^{\frac{n}{2}}$ such that $S_1^{-1}A_1S_1$ is on reduced column echelon form, so we are done with $F=0$.

If there are pivots which are $\leq \varepsilon^{\delta_1}$, let $F_1 \in gl(n, \C)$ whose non zero coefficients are those pivots (then $\Vert F_1 \Vert \leq \varepsilon^{\delta_1}$, and $A_1-F_1$ has a block not of maximal rank). Apply Corollary \ref{coro-ech} to $A_1-F_1$: there exists a unitary matrix $U_1 \in Gl(n, \C)$ such that $A_2 = U_1^{-1}(A_1-F_1)U_1$ is block diagonal, with blocks on column echelon form and by Lemma \ref{lKerSpectrum}, $\sigma \ker (A_2) > \sigma \ker (A_1)$. We have estimates
\[ \Vert F_1 \Vert \leq \varepsilon^{\delta_1}\]
\[ \Vert A_2 \Vert = \Vert U_1^{-1}(A_1-F_1)U_1 \Vert \leq \Vert A_1 \Vert + \varepsilon^{\delta_1} 
.\]

\textbf{Induction step:} Assume that for some $k >1$ there exists a unitary $U'_{k-1} \in Gl(n,\C)$ with ${U'}_{k-1}^{-1}AU'_{k-1} = A_{k}+F'_k$ such that $A_{k}$ is block diagonal and has blocks on column echelon form, with $\Vert F'_{k} \Vert \leq \varepsilon^{\delta_{k-1}} + \dots + \varepsilon^{\delta_1}$ and $\Vert A_k \Vert \leq \Vert A \Vert+\varepsilon^{\delta_1}+\dots +\varepsilon^{\delta_{k-1}}$.

If no pivot of $A_k$ is $\leq \varepsilon^{\delta_{k}}$, applying Lemma \ref{lScaling} to each block of $A_k$, we get $S_k \in Gl(n, \C)$ diagonal, with $\Vert S_k^{\pm 1} \Vert \leq C (\frac{\Vert A_k \Vert}{\varepsilon^{\delta_k}})^{\frac{n}{2}}$, and $A'=S_k^{-1} A_k S_k$ is block diagonal with blocks on reduced column echelon form. Let then $S=U'_{k-1}S_k$, then
\[ S^{-1}AS = (U'_{k-1}S_k)^{-1}A(U'_{k-1}S_k) =S_k^{-1}(A_k+F'_k)S_k\]
\[=A' + F'\]
with $F' \in gl(n, \C)$ of norm
\[ \Vert F' \Vert \leq \Vert S_k \Vert\ \Vert S_k^{-1}  \Vert (\varepsilon^{\delta_{k-1}} + \dots + \varepsilon^{\delta_1}) \leq 2C^2(\frac{\Vert A_{k}\Vert}{\varepsilon^{\delta_k}})^n\varepsilon^{\delta_{k-1}} \leq 2C^2(\Vert A \Vert +2\varepsilon^{\delta_{k-1}})^n \varepsilon^{-n\delta_k+\delta_{k-1}}\]
where the constant $C$ depends only on $n$, and the proposition is proved.

If some pivots of $A_k$ are $\leq \varepsilon^{\delta_k}$, let $F_k \in gl(n, \C)$ whose non zero coefficients are those pivots (then $\Vert F_k \Vert \leq \varepsilon^{\delta_k}$, and $A_k - F_k$ is block diagonal and has an block not of maximal rank). Apply Corollary \ref{coro-ech} to $A_k - F_k$: there exists a unitary matrix $U_k \in Gl(n,\C)$ such that $A_{k+1} = U_k^{-1} (A_k-F_k)U_k$ is block diagonal and has blocks in column echelon form and by Lemma \ref{lKerSpectrum}, $\sigma \ker (A_{k+1}) > \sigma \ker (A_k)$ with estimates 
\[ \Vert F_{k} \Vert \leq \varepsilon^{\delta_k},\]
\[\Vert A_{k+1} \Vert \leq  \Vert A_k \Vert +\varepsilon^{\delta_k} \]

\noindent therefore 

$$U_k^{-1} U_{k-1}^{'-1} A U'_{k-1}U_k = U_k^{-1}(A_k+F'_k) U_k
=A_{k+1} + U_k^{-1} (F_k+F'_k)U_k 
$$

\noindent and then, letting $U'_k = U'_{k-1}U_k$ and $F'_{k+1} = F_k+F'_k$, the induction step is established. By construction, the constructed matrices have the same block decomposition as $A$.

\bigskip
After at most $\frac{n^2}{2}$ steps, the algorithm stops according to \ref{algostop}, since the matrix is zero and is then trivially in normal Jordan form.
\end{proof}
\subsection{From reduced  echelon to Jordan}

We shall  conjugate a (nonzero) nilpotent matrix $A \in \mathcal{M}(n, \C)$ on reduced column echelon form to Jordan normal form with control on the conjugation. If $A$ is on reduced column echelon form, iterating the following lemma will remove the non zero coefficients above the pivots of $A$. We will first give a Lemma to remove one of these coefficients, and then the others will be removed by iteration. \\
We denote in the following $A_{i,\cdot}$ the $i$-th row of $A$, and $A_{\cdot,l}$ the $l$-th column of $A$. 
\medskip

Fix $1\le k_0< i_0< j_0$ and define $M:=M_{k_0,i_0}\in \mathcal{M}(n,\C)$ by
$$M_{i,j}=\begin{cases} 0 &   (i,j)\neq (k_0,i_0)  \\ 1 &   (i,j)=(k_0,i_0) \end{cases}.$$

Multiplying a matrix $A$ to the left by $I+aM$ amounts to replacing  the $k_0$-th row $A_{k_0, \cdot}$ by $A_{k_0, \cdot}+a A_{i_0, \cdot}$.

Multiplying a matrix $A$ to the right by 
$$(I+aM)^{-1}=I-aM$$ 
amounts to replacing  the $i_0$-th column  $A_{\cdot, i_0}$ by $A_{\cdot, i_0}-a A_{\cdot, k_0}$.

\medskip

Let $A$ be a nilpotent matrix 
on reduced column echelon form with a pivot
$$A_{i_0,j_0},\quad i_0< j_0.$$
Let  $k_0\in \{1,\dots,  i_0-1\}$ the row index of the last non zero coefficient before the pivot in the $j_0$-th column of $A$, let $M$ as above,
and
$$B=(I-aM)A(I+aM),\quad a=A_{k_0,j_0}.$$

\noindent The following lemma will be used to remove that coefficient ($a=A_{k_0,j_0}$) from the matrix $A$.

\begin{lemma} \label{pivot} The matrix $B$ defined above is on reduced column echelon form with the same  pivots as $A$.

If the only non-zero coefficients in $A_{\cdot, j}$, $j>j_0$ are the pivots, then

$$B_{\cdot, j}=A_{\cdot, j},\quad j>j_0$$
and
$$
B_{i,j_0}=\begin{cases} A_{i,j_0} &  \mathrm{if}\ i \neq k_0  \\ 0 &  \mathrm{if}\  i =k_0  \end{cases}$$ 
\end{lemma}

\begin{proof} We have
$$\forall j\neq i_0,\ B_{k_0, j}=A_{k_0, j}-aA_{i_0, j}\quad\ \mathrm{and} \ \quad \forall i\neq k_0,\ B_{i, i_0}=A_{i, i_0}+aA_{i, k_0}$$
and $B_{k_0,i_0}=A_{k_0,i_0}$,
and for $i\not= k_0$ and $j\not= j_0$,
$$B_{i,j}=(A-aMA+aAM-a^2MAM)_{i,j}=A_{i,j}.$$

Since $k_0<i_0$, the column $A_{\cdot, k_0}$ is strictly shorter than $A_{\cdot, i_0}$. The pivot in $B_{\cdot, i_0}$ is therefore the same as in $A_{\cdot, i_0}$.

Since $A_{i_0, j_0}$ is a pivot, we have  $A_{i_0, j}=0 $ for all $j<j_0$. Hence
$$B_{k_0,j}=A_{k_0,j}-aA_{i_0,j}= A_{k_0,j},\quad \forall j<j_0.$$
So if there is a pivot in $B_{k_0,\cdot}$, it is the same as that in $A_{k_0,\cdot}$.

\medskip

If the only non-zero coefficients in $A_{\cdot,j}$, $j>j_0$, are the pivot, then, since $k_0 <j_0$ and $i_0 < j_0$,
$$B_{k_0,j}=A_{k_0,j}-aA_{i_0,j}=0\quad \forall j>j_0.$$
Hence
$B_{\cdot,j}=0$ for all $j>j_0$. Moreover,
$$
B_{i,j_0}=\begin{cases} A_{i,j_0} &  \mathrm{if}\  i \neq k_0  \\  B_{k_0,j_0}=A_{k_0,j_0}-aA_{i_0,j_0}=0&   \mathrm{if}\  i =k_0  \end{cases}.$$ 

\end{proof}

\medskip

Now we can conjugate to the Jordan normal form by iterating Lemma \ref{pivot}:

\begin{proposition}\label{red-JNF} 
Let $A$ be a non zero block diagonal nilpotent matrix, with each block on reduced column echelon form. There exists $S\in Gl(n,\C)$,
$$\Vert S^{\pm 1}\Vert\le C(1+\Vert A\Vert)^{n!},$$
such that $S^{-1}AS$
is on Jordan normal form. Moreover $S$ is also block diagonal with the same block decomposition as $A$.
The constant $C$ only depends on $n$. 

\end{proposition}

\medskip

\begin{proof} We start with the column $A_{\cdot,n}$ and apply the lemma \ref{pivot} to each coefficient above the pivot. This gives a $S_1\in Gl(n,\C)$,
$$\Vert S_1^{\pm 1}\Vert\le C_1 (1+\Vert A\Vert)^{n-1},$$
such that
$$A_1=S_1^{-1}AS_1$$
is block diagonal with each block on reduced column echelon form with the same pivots as $A$, and whose only non-zero coefficient in the last column is the pivot. Recall the pivots are equal to 1.

Then we do the same with the next to last column in $A_1$, and so on and so forth. This stops after at most $n$ steps producing a $S\in Gl(n,\C)$,
$$\Vert S^{\pm 1}\Vert\le C_n (1+\Vert A\Vert)^{n!},$$
such that
$$B=S^{-1}AS$$
is block diagonal with blocks on reduced column echelon form with the same pivots as $A$ (they are equal to 1), and whose only non-zero coefficient are the pivots. By construction the block decomposition is the same as $A$.

Since there are only finitely many matrices with only $0$ or $1$ as their coefficients, they can be conjugated to Jordan normal form with a uniform bound on the conjugation and its inverse.

\end{proof}

\subsection{Jordan normal form with estimates}

\begin{proposition}\label{JNF-estimates}
Let $N$ be a non zero nilpotent block diagonal matrix.
For any $\varepsilon\in (0,1)$ and $m \in \N^{\ast}$, there exists $S\in Gl(n,\C)$ block diagonal, a constant  $C>0$ only depending on $n$ and constants $c>0,c'\in (0,\frac{1}{2}]$ only depending on $n$ and $m$ (in particular, they do not depend on $\varepsilon$) such that  

\begin{equation}\label{estim-S} \Vert S^{\pm 1}\Vert \leq  C(\Vert N\Vert+1)^{ c} \varepsilon^{-\frac{1}{2(m+2)} }\end{equation}
    
   \noindent 
\[ S^{-1}N S =  A'+  F',\]
with $A'$ on Jordan normal  form and
\begin{equation}\label{estim-F}
    \Vert S^{\pm 1}\Vert^m \Vert  F' \Vert\le (C(\Vert N\Vert+1)^c)^{m+1}\varepsilon^{c'}.
\end{equation}

\noindent Moreover, $S$ and $A'$ have the same block diagonal decomposition as $N$.

    \end{proposition}
 
 \medskip
 
\begin{proof} 
Let $\varepsilon\in (0,1)$ and $m\geq 1$. Let $\delta_0=1$, $c_n =2(n+1)(2n)!+n$, and for all $ k\ge 1$,
$$\delta_{k-1} =  2(m+2)c_n \delta_{k}.$$
Then $(\delta_k)$ is a positive decreasing sequence, and it satisfies inequality \eqref{decroissantedelta} of Proposition \ref{echelonnee}. Apply Proposition \ref{echelonnee} to $N$ with this choice of $(\delta_k)$.
This gives an integer $k$ with $1\le k\le n^2$, a constant $C$ only depending on $n$, and a block diagonal matrix $S_1\in Gl(n,\C)$,
      \[ \Vert S_1^{\pm 1}\Vert \leq  C(\Vert N\Vert+1)^{\frac n2}  \varepsilon^{-\frac n2 \delta_k} \]
  such that
\[ S_1^{-1}N S_1 =  N'+  G,\]
with $N'$ on reduced column echelon form and
\[\Vert  G \Vert \leq C(\Vert N\Vert+1 )^{n} \varepsilon^{- n \delta_{k}+\delta_{k-1} } .\]
Moreover, $S_1$ and $N'$ are block diagonal with the same block structure as $N$.

\noindent Apply now Proposition \ref{red-JNF} to each block of $N'$ to find a block diagonal
 $S_2\in Gl(n,\C)$ satisfying
$$\Vert S_2^{\pm 1}\Vert\le C(1+\Vert N'\Vert)^{n!}$$
such that the block diagonal matrix $A'=S_2^{-1}N'S_2$
is on Jordan normal form.
Then we have, with $S=S_1S_2$,
$$S^{-1}NS=S_2^{-1}(N'+G)S_2=A'+S_2^{-1}GS_2.$$
Moreover, $S$ and $A'$ are block diagonal with the same block structure as $N$.
\medskip

{\it The estimates:}

$$\Vert N'\Vert= \Vert  S_1^{-1}NS_1-G\Vert \le $$
$$\le  C(\Vert N\Vert+1)^{ n}  \varepsilon^{- n \delta_{k}} \Vert  N\Vert+ C(\Vert N\Vert+1 )^{n} \varepsilon^{- n \delta_{k}+\delta_{k-1} } $$
$$\le   C(\Vert N\Vert+1)^{ n+1} \varepsilon^{- n \delta_{k}} $$

\noindent 
Hence
$$\Vert S_2^{\pm 1}\Vert\le C(\Vert N\Vert+1)^{ (n+1)n!} \varepsilon^{-n\cdot n! \delta_{k}}.$$

\noindent 
In particular
$$\Vert S^{\pm 1}\Vert = \Vert (S_1S_2)^{\pm 1} \Vert \leq  C(\Vert N\Vert+1)^{ (n+1)!+\frac n2} \varepsilon^{-n\cdot n! \delta_{k} -\frac n2 \delta_{k}}\leq C(\Vert N\Vert +1)^{c_n}\varepsilon^{-c_n\delta_k}
$$

\noindent 
(recall $c_n =2(n+1)(2n)!+n$)
and by the definition of $(\delta_k)$, $c_n\delta_k\leq \frac{1}{2(m+2)}$
so \eqref{estim-S} holds. Moreover,

$$\Vert S_2^{-1}GS_2\Vert  \le  
C(\Vert N\Vert+1)^{ 2(n+1)!+n} \varepsilon^{-2n\cdot n! \delta_{k}- n \delta_{k}} \varepsilon^{\delta_{k-1} } $$
$$\le  
C(\Vert N\Vert+1)^{c_n} \varepsilon^{-c_n\delta_{k}+\delta_{k-1} }.$$

\noindent 
Hence
$$\Vert S^{\pm 1}\Vert^m\Vert S_2^{-1}GS_2\Vert  \le C^{m+1}(\Vert N\Vert+1)^{(m+1)c_n} \varepsilon^{-(m+1)c_n \delta_{k} +\delta_{k-1} } $$

$$\leq (C(\Vert N\Vert +1)^{c_n})^{m+1} \varepsilon^{(m+1)c_n \delta_k},$$

\noindent so \eqref{estim-F} holds with $c=c_n$ and $c'=(m+1)c_n\delta_k$.
\end{proof}

\section{An almost reducibility result to a Jordan normal form}\label{section4bis}

In this section, we construct an almost conjugation to a cocycle which is in Jordan normal form. A control of the estimates requires to do it in two steps.

\begin{proposition}\label{ARtoblocks}
    Let $X_{\omega,A}$ an almost reducible cocycle in $\mathcal{C}^\infty$, i.e. there exist sequences $(Z_j)$, $(B_j)$ and $(F_j)$ satisfying conditions \eqref{eq1.4} and \eqref{conditionconvergence}. \\
    If  $\omega$ is Diophantine, then there exist sequences $(\hat Z_j)$, $(\hat F_j)$ and $(\hat B_j)$ such that every $\hat B_j$ is block diagonal, each block being upper triangular with only one eigenvalue, 
    
 \[ \partial_\omega \hat Z_j = A \hat Z_j - \hat Z_j(\hat B_j + \hat F_j)\]
    \noindent and for all $m\in\N$,

    \begin{equation}\label{convergence-blocks}
    \Vert \hat Z_j^{\pm 1}\Vert^m_{\CC^m}\Vert \hat F_j\Vert_{\CC^m}\rightarrow 0.
    \end{equation}
\end{proposition}

\begin{proof}
We will construct such a sequence for a fixed $m$, and then the lemma \ref{convergence-m} will imply the conclusion. 

\bigskip
\noindent Fix $m\in \N \backslash \{0\}$ and define the parameters 

$$\beta = \frac{1}{4n^3},
\gamma_1=\frac{1}{4(16mn^3)^n},$$

\noindent and for $i\geq 1$, the increasing sequence 

$$\gamma_{i+1} = 16mn^3 \gamma_i.$$

\paragraph{
Estimate of \texorpdfstring{$B_j$}{}}

The conjugation relation given by the almost reducibility assumption can be written as
\[B_j = Z_j^{-1}(AZ_j -\partial_\omega Z_j )-F_j\]
and then 
\[\Vert B_j\Vert\leq C_{A, \omega, d} \Vert Z_j\Vert_{\CC^1} \Vert Z_j^{-1}\Vert_{\CC^0} + \Vert F_j\Vert_{\CC^r} \leq C'_{A, \omega, d} \Vert Z_j\Vert_{\CC^1} \Vert Z_j^{-1}\Vert_{\CC^0},\]
with $C_{A, \omega, d}$ and $C'_{A, \omega, d}$ depending only on $A$, $\omega$ and $d$. 

\noindent 
The assumption
\[ \lim_{j \rightarrow + \infty} \Vert Z_j^{\pm 1} \Vert_{\CC^r}^m \Vert F_j \Vert_{\CC^r} = 0, \quad \forall  r,m \in \N\]
implies that for $J_1$ large enough, for all $j \geq J_1$,

\[(\Vert Z_j\Vert_{\CC^m}\Vert Z_j^{- 1}\Vert_{\CC^m})^{\frac{1}{\beta}}\Vert F_j \Vert_{\CC^m} \leq {C'_{A,\omega, d}}^{-\frac{1}{\beta}}\]
therefore for all $j \geq J_1$,
\begin{equation}\Vert B_j \Vert \leq \Vert F_j\Vert_{\CC^0} ^{-\beta}. \label{def-J1}\end{equation}

\noindent 
From now on, we will work with $j \geq J_1$.
\paragraph{Construction of a matrix similar to \texorpdfstring{$B_j$}{}, block diagonal with separated spectrum}

\bigskip
Let $\varepsilon_j =\Vert F_j\Vert_{\CC^m}$.
    From Corollary \ref{induction-blocs} given in appendix and applied to $B_j$ and $\Gamma_i =\varepsilon_j^{\gamma_i}$, there exists $d_0\in \{1,\dots, n\}$ such that 
    we can conjugate $B_j$ to  a matrix $D_j$ which is block diagonal, with each block being upper triangular with $\varepsilon_j^{\gamma_{d_0}}$-connected spectrum (as defined in Section \ref{def-gamma-close}), and if we denote by $M_j\in GL(n,\mathbb{C})$ the conjugation (so that $D_j=M_j^{-1} B_jM_j$ and $G_j =M_j^{-1} F_jM_j $), 
    then 
    
\begin{equation}\label{majoration-de-mu} \Vert M_j \Vert, \Vert M_j^{-1} \Vert \leq 
n^{3n} (\frac{\Vert B_j \Vert}{\varepsilon_j^{2\gamma_{d_0-1}}})^{n^3}.
\end{equation}

\noindent Write $D_j + G_j = \hat B_j + \hat F_j$, where $\hat F_j$ is obtained from $G_j$ by adding a diagonal with coefficients of $D_j$ smaller than $n\Vert F_j\Vert_{\CC^m}^{\gamma_{d_0}}$. Then $\hat B_j$ is block diagonal, with blocks upper triangular having only one eigenvalue (since the difference between close eigenvalues was moved to $\hat F_j$, so the eigenvalues of each block are now the same), and we obtain a conjugation 

$$M_j^{-1} (B_j+F_j)M_j = \hat B_j + \hat F_j
$$

\noindent with the estimate
\[ \Vert \hat F_j \Vert_{\CC^m} \leq 
\Vert G_j\Vert_{\mathcal{C}^m} + n \Vert F_j\Vert_{\mathcal{C}^m}^{\gamma_{d_0}}
\leq \Vert M_j\Vert \Vert M_j^{-1}\Vert \varepsilon_j + n \varepsilon_j^{\gamma_{d_0}}
\leq C_n \varepsilon_j^{1- 2\beta n^3 -4n^3 \gamma_{d_0-1}}+ n \varepsilon_j^{\gamma_{d_0}}
\leq C_n\varepsilon_j^{\gamma_{d_0}}\]

\noindent (where $C_n$ only depends on $n$). Let $\hat Z_j = Z_jM_j$, then 

$$\Vert \hat Z_j^{\pm 1} \Vert_{\CC^m}^m  \Vert \hat F_j \Vert_{\CC^m} \leq \Vert Z_j^{\pm 1}\Vert^m_{\CC^m}\Vert M_j\Vert^m \Vert M_j^{-1}\Vert ^m \Vert \hat F_j\Vert _{\CC^m}\leq 
\Vert Z_j^{\pm 1}\Vert^m_{\CC^m}n^{6mn} (\frac{\Vert B_j \Vert}{\varepsilon_j^{2\gamma_{d_0-1}}})^{2mn^3}
C_n\varepsilon_j^{\gamma_{d_0}}
$$

$$\leq C(m,n)\Vert Z_j^{\pm 1}\Vert^m_{\CC^m}
\varepsilon_j^{-4mn^3\gamma_{d_0-1}-2mn^3\beta +\gamma_{d_0}}
$$

\noindent By the choice of $\beta$ and the sequence $\gamma_i$, the exponent on $\varepsilon_j$ is positive, therefore the almost reducibility assumption implies that this quantity tends to $0$.
\end{proof}

\begin{proposition}\label{almost-jordan}
      Let $X_{\omega,A}$ an almost reducible cocycle in $C^\infty$. If  $\omega$ is Diophantine, then $X_{\omega,A}$ is almost reducible to a sequence $(B_j)$ of matrices that are in Jordan normal form.
\end{proposition}

\begin{proof}
    By Proposition \ref{ARtoblocks}, there exist sequences $(Z_j),(F_j),(B_j)$ such that 

    $$\partial_\omega Z_j = AZ_j -Z_j (B_j+F_j)
    $$

\noindent 
    where the convergence condition \eqref{convergence-blocks} holds, and the matrices $B_j$ are block diagonal, every block being upper triangular with only one eigenvalue.

    \bigskip
\noindent     Fix $m \in \N$. Define $\beta = \frac{c'}{2c(m+2)}$ where $c>0, c'\in (0,\frac{1}{2}]$ are given by the proposition \ref{JNF-estimates} and depend only on $n,m$.

\noindent     Reasoning similarly as in the proof of Proposition \ref{ARtoblocks}, there exists an index $J_2$ such that if $j\geq J_2$, then

    $$\Vert B_j\Vert \leq \Vert F_j\Vert _{\CC^m}^{-\beta}.
    $$

\noindent     We shall apply Proposition \ref{JNF-estimates} with $N$ being the matrix $B_j$ without its diagonal coefficients, and $\varepsilon =\varepsilon_j=\Vert F_j\Vert_{\CC^m}
$.

\noindent      Let $A'=A'_j,F'=F'_j, S=S_j$ 
     given by proposition \ref{JNF-estimates} (in particular $S_j$ block diagonal with the same block structure as $B_j$).
   Therefore, since the diagonal part of $B_j$ commutes with $S_j$ (recall that $B_j$ has only one eigenvalue for each diagonal block), $B_j  + F_j$ is conjugate via $S_j$ to $\tilde B_j + \tilde F_j$, such that $\tilde B_j$ is in Jordan normal form and $\tilde F_j = F'_j + S^{-1}_j F_j  S$ with 
   
\begin{equation} \Vert S_j^{\pm 1}\Vert \leq  C(\Vert B_j\Vert+1)^{ c} \varepsilon_j^{-\frac{1}{2(m+2)} }\leq C(m,n)\varepsilon_j^{-\beta c -\frac{1}{2(m+2)}}
\end{equation}
    
   \noindent and
\begin{equation}
    \Vert S_j^{\pm 1}\Vert^m \Vert  F_j' \Vert\le (C(\Vert B_j\Vert+1)^c)^{m+1}\varepsilon_j^{c'}
    \leq C(m,n)\varepsilon_j^{-c\beta (m+1)+c'}.
\end{equation}

Therefore
\begin{equation}\label{estimation-tildeF_j} 
    \begin{split}
        \Vert S_j^{\pm 1}\Vert^m \Vert \tilde F_j \Vert_{\CC^m} & \leq  C(m,n)\varepsilon_j^{-c\beta (m+1)+c'} + 
         \Vert S_j^{\pm 1}\Vert^{m+2}\Vert  F_j\Vert_{\CC^m} \\
         & \leq C(m,n)\varepsilon_j^{-c\beta (m+1)+c'} + 
      C(m,n)\varepsilon_j^{-(m+2)(\beta c+\frac{1}{2(m+2)})+1} \\
    \end{split}
\end{equation}
Thus 

$$\Vert Z_j^{\pm 1}\Vert_{\CC^m}^m 
\Vert S_j^{\pm 1} \Vert ^m
 \Vert \tilde F_j\Vert _{\CC^m}
$$
$$\leq C(m,n)\Vert Z_j^{\pm 1}\Vert_{\CC^m}^m ( \varepsilon_j^{-c\beta (m+1)+c'} + 
      \varepsilon_j^{-(m+2)(\beta c+\frac{1}{2(m+2)})+1})
$$

The choice of the parameter $\beta$ implies that the exponent on $\varepsilon_j$ on the right hand side will be positive, thus the convergence condition holds for fixed $m$. Applying Lemma \ref{convergence-m}, there is almost reducibility to the sequence $(\tilde B_j)$.
\end{proof}

\section{Construction of a conjugation to a real matrix}\label{section5}

\medskip

In this section, $B= \mathrm{diag}(B_j)_{j=1}^l $ will be a  block diagonal matrix where each block
$B_j$ is on Jordan normal form with only one eigenvalue $\alpha_j$.
The spectrum of $B$ is
$$\sigma(B)=\{\alpha_1,\dots,\alpha_l\}\quad (\#\sigma(B)=l).$$

\noindent 
We shall study the equation
\begin{equation}\label{conjug2}
F = \partial_\omega V -  BV + V\bar{ B}.\end{equation}
where $V : \T^d \rightarrow Gl(n,\C)$  and $\omega\in \mathcal{DC}(\kappa,\tau)$, and where $F$ is supposed to be ``small''.

\noindent 
This equation decomposes into its block-components
\begin{equation}\label{conjug2bis}
F_i^j = \partial_\omega V_i^j  -  B_iV_i^j  + V_i^j \bar{ B}_j\end{equation}
for each $i,j \in \llbracket 1,l\rrbracket$.

\subsection{$(N,\rho)$-linkedness} 

In this section, we will study resonances between the eigenvalues of the matrix $B$. This way, we want to create sets of eigenvalues with same cardinality, linked by resonances. 
Fix $N\in\mathbb{N},\rho>0$.

\medskip

\begin{definition}
Two complex numbers $\alpha$ and $\beta$ are $(N,\rho)- \mathrm{linked}$   if and only if 
$$\vert 2i\pi \langle k , \omega \rangle -(\alpha - \bar \beta )\vert< \rho$$
for some $\vert k\vert\le N$.
\end{definition}

\medskip

Let $\Gamma$ be a finite set of complex numbers with $0<\#\Gamma\le n$.

\medskip

\begin{definition}
An $(N,\rho)- \mathrm{chain}$  in $\Gamma$ of length $r-1$  is a sequence
$$\alpha_1, \alpha_2,\dots, \alpha_{r-1},\alpha_r$$
in $\Gamma$  such that $\alpha_j$ and $\alpha_{j+1}$ are $(N,\rho)$- linked for all $j$. 
The numbers $\alpha_1$ and $\alpha_r$ are then said to be $(N,\rho)-\mathrm{chain-linked}$.

\noindent 
An $(N,\rho)-\mathrm{loop}$ in $\Gamma$ is an $(N,\rho)$- \textit{chain}
$$\alpha_1, \alpha_2,\dots, \alpha_{r-1},\alpha_r$$
such that $\alpha_1=\alpha_{r}$ . An $(N,\rho)-\mathrm{loop}$ is $\textrm{odd}$ if it is of odd length.
\end{definition}

\medskip

\noindent 
It is easy to verify that if $\alpha$ and $\beta$ are $(N,\rho)$-chain-linked, then they are $(N,\rho)$-chain-linked by a chain of length $\le n$.

\medskip

\begin{lemma}\label{sublemma}
Let
$$\alpha_1, \alpha_2,\dots, \alpha_{r-1},\alpha_r$$
be an  $(N,\rho)- chain$  in $\Gamma$.

\begin{itemize}
\item
If $r-1$ is odd, then $\alpha_1$ and $\alpha_r$ are   $(nN,n\rho)-linked$.
\item
If $r-1$ is even, then $\alpha_1$ and $\bar \alpha_r$ are   $(nN,n\rho)-linked$.
\end{itemize}
\end{lemma}

\medskip

\begin{proof} We can assume without restriction that $r-1\le n$.

\noindent 
We have for all $j$ and for some $k_j$,
$$\vert 2i\pi\langle k_j,\omega\rangle-(\alpha_j-\bar\alpha_{j+1})\vert\le \rho,\quad \vert k_j\vert\le N,$$
Let $k= \sum_{j\ \mathrm{odd}}k_j-  \sum_{j\ \mathrm{even}}k_j$, then 
$$\vert k\vert 
\le nN.$$

\noindent 
If $r-1$ is odd, then
$$\alpha_1-\bar \alpha_r=(\alpha_1-\bar\alpha_2)+(\bar \alpha_2-\alpha_3)+(\alpha_3-\bar\alpha_4)+\dots+(\alpha_{r-1}-\bar\alpha_r)$$
so
$$2i\pi\langle k,\omega\rangle-(\alpha_1-\bar\alpha_{r})=\sum_{j\ \mathrm{odd}}\big( 2i\pi\langle k_j,\omega\rangle-(\alpha_j-\bar\alpha_{j+1})\big)+
\sum_{j\ \mathrm{even}}\big( 2i\pi\langle -k_j,\omega\rangle-(\bar \alpha_j-\alpha_{j+1})\big).$$
This implies that
$$\vert 2i\pi\langle k,\omega\rangle-(\alpha_1-\bar\alpha_{r})\vert\le n\rho.$$

\noindent 
If $r-1$ is even, then
$$\alpha_1-\alpha_r=(\alpha_1-\bar\alpha_2)+(\bar \alpha_2-\alpha_3)+(\alpha_3-\bar\alpha_4)+\dots+(\bar \alpha_{r-1}-\alpha_r)$$
so
$$2i\pi\langle k,\omega\rangle-(\alpha_1-\alpha_{r})=\sum_{j\ \mathrm{odd}}\big( 2i\pi\langle k_j,\omega\rangle-(\alpha_j-\bar\alpha_{j+1})\big)+
\sum_{j\ \mathrm{even}}\big( 2i\pi\langle -k_j,\omega\rangle-(\bar \alpha_j-\alpha_{j+1})\big).$$
This implies that
$$\vert 2i\pi\langle k,\omega\rangle-(\alpha_1-\alpha_{r})\vert\le n\rho.$$
\end{proof}

\medskip

\noindent 
We shall assume that $\Gamma$ is such that
\begin{equation} \label{star2}
\text{for any }\alpha\in\Gamma\text{, there is a }\beta \in \Gamma \text{ such that }\alpha\text{ and }\beta\text{ are } (N,\rho)\text{-linked.}
\end{equation}
Then "being $(N,\rho)$-chain-linked'' is an equivalence relation and we denote by $[\alpha]$
the equivalence class of $\alpha\in\Gamma$ (it depends on $(N,\rho)$).

\medskip

\begin{lemma}\label{paires2} Let $\mathcal{E}\subset \Gamma$ be an equivalence class for the relation of being chain-linked.
If $\mathcal{E}$ contains an odd $(N,\rho)$-loop, then any $\beta\in \mathcal{E}$ is $(nN,n\rho)$-linked to itself.

\end{lemma}

\medskip

\begin{proof} Consider an odd $(N,\rho)$-loop 
$$\alpha_1, \alpha_2,\dots, \alpha_{r-1},\alpha_r=\alpha_1,\quad r\ge 2.$$
If $r-1\ge n+1$, then there exist $1\le i< j\le r-1$ such that $\alpha_i=\alpha_j$.
Then
$$\alpha_i, \alpha_{i+1},\dots, 
\alpha_j=\alpha_i$$
is an  $(N,\rho)$-loop of length $j-i$.
Moreover
$$\alpha_j, \alpha_{j+1},\dots, \alpha_{r-1},\alpha_r = \alpha_1, \alpha_2,\dots, \alpha_{i-1},\alpha_ i= \alpha_j$$
is an  $(N,\rho)$-loop of length $(r-j)+(i-1)$.

\noindent 
Since
$$(j-i)+(r-j)+(i-1)=r-1$$
is odd, one of these two ``sub-loops'' must be odd. So there exists a shorter odd $(N,\rho)$-loop, and we conclude that there exists an odd $(N,\rho)$-loop of length $< r-1$,
i.e. we can assume that $r-1\le n$.

\noindent 
Let now $\beta\in \mathcal{E}$. Then there is an $(N,\rho)$-chain connecting $\beta$ to the odd $(N,\rho)$-loop 
$$\alpha_1, \alpha_2,\dots, \alpha_{r-1},\alpha_r=\alpha_1,$$
i.e. 
$$\beta=\beta_1, \beta_2,\dots, \beta_{s-1},\beta_s=\alpha_j,\quad s-1\le n-(r-1)$$

\noindent 
We can without restriction assume that $j=1$.
Then
$$\beta=\beta_1, \dots, \beta_{s-1},\beta_s=\alpha_1, \alpha_2,\dots,\alpha_r=\alpha_1=\beta_s,\beta_{s-1},\dots,\beta_2,\beta_1=\beta$$
is an $(N,\rho)$-loop of length
$$(s-1)+(r-1)+(s-1)  \le 2(n-r+1)+(r-1)= 2n-r+1\le 2n-1.$$

\noindent 
This length is odd, thus by Lemma \ref{sublemma}, $\beta$ is $(nN,n\rho)$-linked to itself.
\end{proof}

\medskip 

\begin{lemma}\label{paires3} 
Let $\mathcal{E}\subset \Gamma $ be an equivalence class for the relation of being chain-linked. If $\mathcal{E}$ contains no odd $(N,\rho)$-loops, then  there exists a partition $\mathcal{E}= \Sigma_1 \cup \Sigma_2$ such that
\begin{itemize}

\item[$(i)$] $\beta$ and $\gamma$  are  not $(N,\rho)$-linked if  $\beta, \gamma\in  \Sigma_1$, and the same holds for $\Sigma_2$,
\item[$(ii)$] $\beta$ and $\bar \gamma$  are  $(nN,n\rho)$-linked if  $\beta, \gamma\in  \Sigma_1$, and the same holds for $\Sigma_2$,
\item[$(iii)$] $\beta$ and $\gamma$  are  $(nN,n\rho)$-linked if  $\beta \in  \Sigma_1$  and  $ \gamma\in  \Sigma_2$.

\end{itemize}
\end{lemma}

\medskip

\begin{proof} 
Let $\alpha\in\mathcal{E}$ be chosen arbitrarily. 
Define
$$\Sigma_1=\{ \beta\in\mathcal{E}: \mathrm {there\ exists\ a }\ (N,\rho) \mathrm {-chain\ of\ even\ length\ between }\  \alpha\ \mathrm{and }\ \beta\}\cup\{ \alpha\}$$
and
$$\Sigma_{2}=\{ \beta\in\mathcal{E}: \mathrm {there\ exists\ a}\ (N,\rho) \mathrm {-chain\ of\ odd\ length\ between }\   \alpha\ \mathrm{and }\ \beta\}.$$

\noindent 
Notice that $\Sigma_2$ contains all elements that are $(N,\rho)$-linked to $\alpha$. In particular,  $\Sigma_{2}\not=\emptyset$, by assumption \eqref{star2} on $\Gamma$.

\noindent 
Notice also  that $\Sigma_{1}\cap \Sigma_{2}=\emptyset$, because if not, then $\mathcal{E}$ would contain an odd $(N,\rho)$-loop.

\medskip

\noindent 
Proof of (i): Two elements in $\Sigma_{1}$ cannot be $(N,\rho) $-linked to each other, because then there would be an odd $(N,\rho) $-loop in $\mathcal{E}$. Idem for $\Sigma_{2}$.

\noindent 
Proof of (ii): Any two elements in $\Sigma_{1}$ are linked by an even $(N,\rho)$-chain (by transversality). Idem for $\Sigma_{2}$. Therefore (ii) follows from Lemma \ref{sublemma}.

\noindent 
Proof of (iii): Any  element in $\Sigma_{1}$ is linked to any element of $\Sigma_{2}$ by a chain of odd length (by transversality). Therefore (iii) follows from Lemma \ref{sublemma}.
\end{proof}

\subsection{Analysis of  resonances}

Now consider the equation \eqref{conjug2} (or equivalently the equations \eqref{conjug2bis}).

\begin{lemma}\label{l6.4} 
Let $i,j\in \llbracket 1,n\rrbracket$ and assume the equation \eqref{conjug2bis} holds. If for some $0<\rho\leq 1$,
$$\vert 2i\pi \langle k , \omega \rangle -(\alpha_i - \bar \alpha_j )\vert\ge \rho,
$$
then
$$\Vert \hat V_i^j(k)\Vert\le (\frac2\rho)^{2n-1}\Vert  F\Vert_{\CC^0}.$$
\end{lemma}

\medskip

\begin{proof} From \eqref{conjug2bis} we have
$$
(2i\pi \langle k , \omega \rangle -(\alpha_i - \bar \alpha_j )) \hat V_i^j(k)=N_i \hat V_i^j(k) -  \hat V_i^j(k)N_j+ \hat F_i^j(k),$$
where
$$N_i=B_i-\alpha_i I\quad\textrm{and}\quad N_j=B_j-\alpha_j I.$$
Changing notations, we can write this as
$$
\gamma X=N_i X-  XN_j+ Y=\mathcal{L}X+Y$$

\noindent (with $X=\hat V_i^j(k),Y=\hat F_i^j(k),\gamma \in \mathbb{C}$).
The operator $\mathcal{L}$ verifies for all matrix $X$ in the domain of $\mathcal{L}$, 
$$\Vert\mathcal{L}X\Vert\le 2\Vert X\Vert\quad\ \mathrm{and}\ \quad \mathcal{L}^{2n-1}=0.$$

\noindent 
Then
$$X=\frac1\gamma\big(Y+\mathcal{L}X)= \frac1\gamma Y+ \frac1{\gamma^2} \mathcal{L}Y+\frac1{\gamma^2} \mathcal{L}^2X+\dots$$
$$=\sum_{0\le j\le 2n-2}\frac1{\gamma^{j+1}}\mathcal{L}^{j}Y.$$
Hence, since $\gamma\geq \rho$,
$$\Vert X\Vert\le \sum_{0\le j\le 2n-2}\frac1{\rho^{j+1}}2^j\Vert Y\Vert\le \frac1{2-\rho} (\frac2{\rho})^{2n-1}\Vert Y\Vert.$$
\end{proof}

\medskip

\begin{lemma}\label{l6.5} Assume \eqref{conjug2} holds. Let $\xi>0$, $N\in \mathbb{N}$, $\rho>0$, $\kappa>0$, $\tau\geq d+1$, and assume $\omega \in DC(\kappa,\tau)$ and 
\begin{enumerate}
\item \label{encadrement-V}
$$\Vert V\Vert_{\CC^0}+\Vert V^{-1}\Vert_{\CC^0}\le \xi,$$
\item

$V$ is a trigonometric polynomial of  degree $\le N,$
\item
$$\rho< (2N)^{-\tau}\kappa,$$
\item \label{small2}
$$\Vert F\Vert_{\CC^0}\le (4^n n! (3N)^d\xi^n)^{-1}\rho^{2n-1}. $$
\end{enumerate}
Then, for any $i$, there exists a $j$ and a unique $\vert k_{i,j}\vert \le N$ such that
$$\vert 2i\pi \langle k_{i,j} , \omega \rangle -(\alpha_i - \bar \alpha_j )\vert< \rho.$$
\end{lemma}

\medskip

\begin{proof}  Given $i$, suppose that
$$\vert 2i\pi \langle k , \omega \rangle -(\alpha_i - \bar \alpha_j )\vert\ge\rho$$
for all $j$ and all $\vert k\vert \le N$. By Lemma \ref{l6.4},
$$\Vert\hat V_i^j(k)\Vert\le  (\frac2\rho)^{2n-1}\Vert  F\Vert_{\CC^0}$$
so
$$\Vert V_i^j\Vert_{\CC^0} \le  (3N)^d(\frac2\rho)^{2n-1}\Vert  F\Vert_{\CC^0}=:\varepsilon.$$
This implies that
$$\Vert \det V\Vert_{\CC^0}\le n! \varepsilon \xi^{n-1}.$$

\noindent 
Since
\begin{equation}\begin{split} \partial_\omega V^{-1} - \bar{  B} V^{-1} + V^{-1} B & = V^{-1}\partial_\omega V V^{-1}- \bar{ B} V^{-1} + V^{-1} B \\
&= V^{-1}( B V - V \bar{ B} + F)V^{-1}-\bar{ B} V^{-1} + V^{-1} B\\
&= V^{-1}FV^{-1}
\end{split}\end{equation}
we find, by a similar reasoning, that
$$\Vert \det( V^{-1})\Vert_{\CC^0}\le n! \varepsilon \xi^2 \xi^{n-1}.$$

\noindent 
Hence
$$1=\Vert \det ( VV^{-1})\Vert_{\CC^0}\le (n!\varepsilon \xi^{n})^2$$
which is forbidden by assumption 4.

\medskip

\noindent 
\textbf{Uniqueness:} Suppose there exists $k\not=l$ such that $\vert k\vert,\   \vert l\vert\le N$ and
$$\vert 2i\pi \langle k , \omega \rangle -(\alpha_i - \bar \alpha_j )\vert<\rho\quad\ \mathrm{and} \ \quad \vert 2i\pi \langle l , \omega \rangle -(\alpha_i - \bar \alpha_j )\vert<\rho.$$
Then, since $\omega \in DC(\kappa,\tau)$,
$$2\pi \frac\kappa{(2N)^{\tau}}\le \vert 2i\pi \langle k -l, \omega \rangle\vert \le 2\rho$$
which is forbidden by assumption 3.
\end{proof}

\medskip

\begin{corollary}\label{cor-equivalence} Under assumptions $1-4$ of Lemma \ref{l6.5}, being $(N,\rho)$-chain-linked is an equivalence relation on $\sigma(B)$.

\end{corollary}

\noindent 
\begin{proof}
    The proof is immediate, since these assumptions imply that the condition \eqref{star2} holds for $\sigma(B)$.
\end{proof}

\subsection{The Jordan structure of $B$}

\textbf{Assumption:} From now on, in this section, we assume properties $1$--$4$ of Lemma \ref{l6.5}.

\medskip

\noindent 
Let 
$$\mathcal{E}\subseteq \sigma (B)$$ 
be an equivalence class (for the relation of being chain-linked) that contains no odd $(N,\rho)$-loop. By Lemma \ref{paires3}
  there exists a partition $\mathcal{E}= \Sigma_1 \cup \Sigma_2$ such that
\begin{itemize}
\item[$(i)$] $\beta$ and $\gamma$  are  not $(N,\rho)$-linked if  $\beta, \gamma\in  \Sigma_1$, and the same holds for $\Sigma_2$,
\item[$(ii)$] $\beta$ and $\bar \gamma$  are  $(nN,n\rho)$-linked if  $\beta, \gamma\in  \Sigma_1$, and the same holds for $\Sigma_2$,
\item[$(iii)$] $\beta$ and $\gamma$  are  $(nN,n\rho)$-linked if  $\beta \in  \Sigma_1$  and  $ \gamma\in  \Sigma_2$.
\end{itemize}

\noindent 
Let $\Sigma_3=\sigma (B)\setminus\mathcal{E}$ and define for any matrix $X$,
$$
X_{\Sigma_u}^{\Sigma_v}=\big(X_i^j\big)_{\alpha_i\in \Sigma_u}^{\alpha_j\in \Sigma_v},\quad u,v=1,2,3,$$
We often write $X_{\Sigma_u}$ for the diagonal block $X_{\Sigma_u}^{\Sigma_u}$.

\medskip

\begin{lemma}\label{l6.7} If
$$\Vert F\Vert_{\CC^0}\le  \frac1{4^n 3 n (3N)^d\xi}\rho^{2n-1}$$
then $V_{\Sigma_1}^{\Sigma_2}$ is an invertible square-matrix and
$$
\Vert (V_{\Sigma_1}^{\Sigma_2})^{-1}-(V^{-1})_{\Sigma_2}^{\Sigma_1}\Vert_{\CC^0}\le  
6n(3N)^d\xi^2(\frac{2}{\rho})^{2n-1}\Vert  F\Vert_{\CC^0}.$$
\end{lemma}

\medskip

\begin{proof} We can assume without restriction that
$$\Sigma_1=\{\alpha_1,\dots,\alpha_r\},\quad \Sigma_2=\{\alpha_{r+1},\dots,\alpha_{r+s}\}$$
so
$$V=\big(U_u^v\big)_{u,v=1,2,3},\quad  U_u^v=V_{\Sigma_u}^{\Sigma_v}.$$

\noindent 
By Lemma  \ref{l6.4}, for any $\vert k\vert\le N$,

$$\Vert \hat V_i^j(k)\Vert \le (\frac2\rho)^{2n-1}\Vert  F\Vert_{\CC^0}=\varepsilon\quad \textrm{if}\ 
\begin{cases} 
 &\alpha_i\in \Sigma_1 \quad \alpha_j\in \Sigma_1\cup \Sigma_3\\
\mathrm{or}& \alpha_i\in \Sigma_2 \quad \alpha_j\in \Sigma_2\cup \Sigma_3\\
\mathrm{or}& \alpha_i\in \Sigma_3 \quad \alpha_j\in \Sigma_1\cup \Sigma_2\\
\end{cases}$$
which implies that
$$\Vert \hat U_u^v(k)\Vert\le n\varepsilon\quad \textrm{if}\ 
\begin{cases}  u=1& v=1,3\\
u=2 & v=2,3\\
u=3& v=1,2.
\end{cases}$$
Hence
$$\Vert U_u^v\Vert_{\CC^0} \le n (3N)^d \varepsilon \quad \textrm{if}\ 
\begin{cases}  u=1& v=1,3\\
u=2 & v=2,3\\
u=3& v=1,2.
\end{cases}$$

\noindent 
Let now
$$W=\begin{pmatrix} 0 & U_1^2 & 0\\ U_2^1 & 0& 0\\ 0&0& U_3^3\end{pmatrix}.$$

\noindent 
Since
$$\Vert W-V\Vert_{\CC^0}
\le 3n(3N)^d\varepsilon = \delta$$
we get that $W$ is invertible and
$$\Vert W^{-1}-V^{-1}\Vert_{\CC^0}\le \sum_{j\ge 1}\Vert  V^{-1} (W-V) \Vert _{\CC^0}^j\Vert V^{-1}  \Vert_{\CC^0}\le \sum_{j\ge 1}(\delta\xi)^j\xi $$
Now if
$$\delta\xi\le \frac12\quad \iff\quad \varepsilon\le  \frac1{6n (3N)^d\xi}\quad \iff\quad \Vert F\Vert_{\CC^0} \le  \frac1{4^n 3 n(3N)^d\xi}\rho^{2n-1}$$
(which holds by assumption),
then
$$
\Vert W^{-1}-V^{-1}\Vert_{\CC^0}\le\frac1{1-\delta\xi}\delta\xi^2\le 2\delta\xi^2 = 6n(3N)^d\varepsilon\xi^2= 6n(3N)^d(\frac2\rho)^{2n-1}\xi^2\Vert  F\Vert_{\CC^0}$$

\noindent 
Finally, by a computation,
$$W^{-1}=\begin{pmatrix} 0 & (U_2^1)^{-1} & 0\\ (U_1^2)^{-1} & 0& 0\\ 0&0& (U_3^3)^{-1}\end{pmatrix},$$
thus the estimate on $(V_{\Sigma_1}^{\Sigma_2})^{-1}-(V^{-1})_{\Sigma_2}^{\Sigma_1}$ holds.
 \end{proof}
 
 \medskip
 
 \begin{lemma}\label{l6.8} There exists a constant $C$ only depending on $n$ and $d$ such that if
  \begin{equation}\label{small3}\begin{cases} 
  \Vert  F\Vert_{\CC^0}<\frac1{CN^d\xi^n}\rho^{2n-1} \\
 \rho <\frac1{CN^d\xi^{n+1}},\end{cases}\end{equation}
then $B_{\Sigma_1}$ and $B_{\Sigma_2}$ have the same Jordan structure, i.e. for any $k\geq 1$,  they have the same number of Jordan blocks of dimension $k$.
\end{lemma}

\medskip

\begin{proof}  From \eqref{conjug2bis} we have
$$(2i\pi \langle k , \omega \rangle -(\alpha_i - \bar \alpha_j )) \hat V_i^j(k)=N_i \hat V_i^j(k) -  \hat V_i^j(k)N_j+ \hat F_i^j(k),\quad \forall 1\leq i,j\leq l,\ \forall k\in\mathbb{Z}^d,$$
where
$$N_i=B_i-\alpha_i I\quad\textrm{and}\quad N_j=B_j-\alpha_j I.$$

\noindent 
If
$$\vert 2i\pi \langle k , \omega \rangle -(\alpha_i - \bar \alpha_j ) \vert \ge\rho,$$
then, by Lemma \ref{l6.4},
$$\Vert \hat V_i^j(k)\Vert\le (\frac2\rho)^{2n-1}\Vert  F\Vert_{\CC^0}.$$
If
$$\vert 2i\pi \langle k , \omega \rangle -(\alpha_i - \bar \alpha_j ) \vert <\rho,$$
then
$$\Vert N_i \hat V_i^j(k) -  \hat V_i^j(k)N_j\Vert\le \rho\xi+\Vert  F\Vert_{\CC^0}.$$
Hence in any case, 
$$\Vert N_i \hat V_i^j(k) -  \hat V_i^j(k)N_j\Vert\le \max( 2(\frac2\rho)^{2n-1}\Vert  F\Vert_{\CC^0},  \rho\xi+\Vert  F\Vert_{\CC^0}  )=\varepsilon$$
and
$$\Vert N_{i} V_{i}^{j} -   V_{i}^{j}N_j\Vert_{\CC^0} \le (3N)^d\varepsilon.$$
This implies that
$$\Vert N_{\Sigma_1} V_{\Sigma_1}^{\Sigma_2} -   V_{\Sigma_1} ^{\Sigma_2}N_{\Sigma_2}  \Vert_{\CC^0} \le n (3N)^d\varepsilon.$$

\noindent 
By Lemma \ref{l6.7}, the bound on $\Vert F\Vert_{\mathcal{C}^0}$ in \eqref{small3} implies 
$$\Vert V_{\Sigma_1}^{\Sigma_2}\Vert_{\CC^0},\ \Vert (V_{\Sigma_1}^{\Sigma_2})^{-1}\Vert_{\CC^0} \ \le\ 2\xi,$$
so, by Proposition \ref{pApp2}, $N_{\Sigma_1}$ and $N_{\Sigma_2}$, hence $B_{\Sigma_1}$ and $B_{\Sigma_2}$,  have the same Jordan structure if
$$
 n (3N)^d\varepsilon< \frac1{n\cdot n! 2^n\xi^n}\quad\iff\quad \varepsilon <\frac1{n^2\cdot n! 2^n(3N)^d\xi^n}.$$
 This holds if
 $$\Vert  F\Vert_{\CC^0}<\frac1{n^2\cdot n! 8^n(3N)^d\xi^n}\rho^{2n-1} $$
 and
$$  \rho <\frac1{n^2\cdot n! 2^{n+1}(3N)^d\xi^{n+1}}.$$
 
\end{proof}

\subsection{Construction of \texorpdfstring{$W$}{} and conjugation to a real matrix.}
\begin{lemma}\label{construction-de-W}
Assume that $V$ is a trigonometric polynomial of degree $N$; let $\xi\geq 0$ such that 

\begin{equation}\label{boundV}\Vert V\Vert _{\mathcal{C}^0} +\Vert V^{-1}\Vert _{\mathcal{C}^0}  \leq \xi.\end{equation}

There exists a constant $C$ which only depends on $n,d$ such that if

\begin{equation}\label{small2bis}\Vert F\Vert_{\mathcal{C}^0} \leq \frac{\rho^{2n-1}}{CN^d\xi^n},
\end{equation}

and 

$$\rho\leq \frac{1}{C} \min(\frac{1}{N^d\xi^{n+1}}, \frac{\kappa}{N^\tau}),
$$
    
    then there exist $W : \T^d \rightarrow Gl(n, \C)$ of class $\mathcal{C}^\infty$, $B' \in gl(n,\R)$ and $B'' \in gl(n,\C)$ such that
    \begin{enumerate}
    \item \label{item0}
\[ \partial_{\frac{\omega}{2}}W =  BW - W(B'+B''),\]
\item $W$ commutes with $B,B' $ and $B''$, and $B'$ has the same diagonal block structure as $B$,
\item \label{item01}
\[ \Vert B'' \Vert  \leq 2n\rho, \]
    \item \label{item1}
    for all $r\in\mathbb{R}$,
\[ \Vert W^{\pm 1} \Vert_{\CC^r} \leq (4n\pi N)^{r},\]
\item \label{item3} there is the estimate 
\[\Vert B'\Vert \leq C\Vert B\Vert.\]
\end{enumerate}
\end{lemma}

\begin{proof}
    By Corollary \ref{cor-equivalence}, $\sigma(B)$ satisfies the property  \eqref{star2}.
 Let $\mathcal{E} \subset \sigma(B)$ be an equivalence class. 
\begin{itemize}
    \item Case 1: if $\mathcal{E}$ contains an odd $(N,\rho)$-loop, then by Lemma \ref{paires2}, for all $\alpha_i \in \mathcal{E}$, there exists $k_{i} \in \Z^d$ with $\vert k_{i} \vert \leq nN$ such that 
     \[\vert \alpha_i - \bar \alpha_i - 2i\pi \langle k_{i}, \omega \rangle \vert \leq n\rho
     .\]
   
    \item Case 2: if $\mathcal{E}$ does not contain any odd $(N,\rho)$-loop, let $\mathcal{E} = \Sigma_1 \cup \Sigma_2$ be the partition given by the Lemma \ref{paires3}. Choose arbitrarily $\alpha_0 \in  \Sigma_1 \cap \mathcal{E} $.
    Then 
    for all $\alpha_i \in \Sigma_2\cap \mathcal{E}$, by Lemma \ref{sublemma}, there exists $k_{i} \in \Z^d$, $\vert k_{i} \vert \leq n N$ such that
    \[ \vert \alpha_i - \bar \alpha_0 -2i\pi \langle k_{i}, \omega \rangle \vert \leq n \rho.\]
    And
    for all $\alpha_i \in \Sigma_1\cap \mathcal{E}$, there exists $k_{i} \in \Z^d$, $\vert k_{i} \vert \leq  nN$ such that
    \[ \vert \alpha_i -  \alpha_0 -2i\pi \langle k_{i}, \omega \rangle \vert \leq n \rho.\]

\noindent and if $\alpha_i=\alpha_0$ then $k_i=0$; 
    
 %
\end{itemize}

Let $W \in \CC^0(\mathbb{T}^d,Gl(n, \C))$ the diagonal matrix whose diagonal coefficients are $(w_j)$ defined as follows: 
given an eigenvalue $\alpha_i \in \sigma(B)$ associated with $k_i$ as defined before, if $\alpha_i$ appears on line $j$ in $B$, then \[ w_j(\theta) = e^{2i\pi \langle k_{i}, \theta \rangle}I\] if $[\alpha_i]$ is in case (1) and \[w_j(\theta) = e^{4i\pi \langle k_{i}, \theta \rangle}I\] 
if $[\alpha_i]$ is in case (2).


Then it holds that
\[\partial_{\frac{\omega}{2}}W =BW-W(B'+B'') \]

\noindent
where the coefficients of $B'+B''$ will be defined as follows:
\begin{itemize}
    \item if $[\alpha_i]$ is in case (1), since
    \[\alpha_i-i\pi \langle k_i, \omega \rangle  = \frac{1}{2}(\alpha_i+\bar \alpha_i) + \frac{1}{2}(\alpha_i-\bar \alpha_i - 2i\pi \langle k_i,\omega \rangle),\]
    one can define $\Re\alpha_i$ as the coefficient of $B'$ and $\frac{1}{2}(\alpha_i - \bar \alpha_i - 2i\pi \langle k_i, \omega \rangle)$ as the coefficient of $B''$;
    \item if $[\alpha_i]$ is in case (2), if
     $\alpha_i \in \Sigma_2 \cap \mathcal{E}$,
     \[ \alpha_i-2i\pi \langle k_i, \omega \rangle  =  \bar \alpha_0 +( \alpha_i - \bar \alpha_0 - 2i\pi \langle k_i, \omega \rangle) \]
    Then $\bar \alpha_0$ is the coefficient of $B'$ and $ \alpha_i - \bar \alpha_0 - 2i\pi \langle k_i, \omega \rangle$ is the coefficient of $B''$. If
     $\alpha_i \in \Sigma_1 \cap \mathcal{E}$,
    \[ \alpha_i-2i\pi \langle k_i, \omega \rangle  =  \alpha_0 +(\alpha _i- \alpha_0-2i\pi \langle k_i, \omega \rangle) \]
    Then $\alpha_0$ is the coefficient of $B'$ and $ \alpha_i - \alpha_0 - 2i\pi \langle k_i, \omega \rangle$ is the coefficient of $B''$.
\end{itemize}

\noindent
Therefore, if $\alpha$ is an eigenvalue of $B'$, then $\bar \alpha$ is also an eigenvalue with the same multiplicity. Moreover, by Lemma \ref{l6.8}, the blocks with eigenvalues $\alpha$ and $\bar\alpha$ have the same Jordan structure. Thus, by Lemma \ref{spectre-conjugué}, one can assume up to a unitary transformation that $B'$ is in real Jordan normal form.

\noindent With our choices of the $k_i$, the values of $W$ commute with $B$, $B'$ and $B''$. The norm of $W$ and the norm of $W^{-1}$ follows from the fact that $\vert k_i \vert \leq nN$.

\bigskip
The matrix $B''$ is bounded by $n\rho$.
Moreover, the coefficients of the matrix $B'$ are the same as those of $B$ outside the diagonal. The diagonal coefficients of $B'$ are in $\sigma(B)\cup \overline{\sigma(B)}\cup \Re \sigma(B)$, 
which implies
\[ \Vert B' \Vert \leq C(n) \Vert B \Vert.\]
\end{proof}

\subsection{Application of the main lemma \ref{construction-de-W}}

To apply the result of the previous section, it will be necessary to have an estimate of the truncation of an application $U$ satisfying $U^{-1}=\bar U$, which will be obtained in the following lemma.

\begin{lemma}\label{troncation}
Let $N \in \N$. Let $U : \T^d \rightarrow Gl(n,\C)$ of class $\CC^\infty$  such that
\[ U^{-1} = \bar U,\]
and let
\[V(\theta) = \mathcal{T}_N U(\theta) := \sum_{\vert k \vert \leq N} \hat U(k)e^{2i\pi \langle k, \theta \rangle}.\]
There exists a constant $C_d > 0$ depending only on $d$ such that if $N$ satisfies 
\begin{equation}  C_d \Vert U \Vert_{\CC^{d+1}}\Vert U \Vert_{\CC^0}\leq N, \label{taillesigman}
\end{equation}
then
\begin{equation}\label{conclusion-lemme-Vbar} \Vert V^{-1} -\bar V \Vert_{\CC^0} \leq \frac{1}{4}\Vert V \Vert_{\CC^0}.
\end{equation}
\end{lemma}

\begin{proof}
We have 
\[ \Vert \hat U (k) \Vert \leq (\frac{1}{2\pi \vert k \vert})^{d+1} \Vert U \Vert_{\CC^{d+1}}, \quad \forall k \in \Z^d,\]
which implies
\begin{equation}\label{V-U-sigma_N}\Vert V-U \Vert_{\CC^{0}} \leq \sum_{\vert k \vert > N} \Vert \hat U (k) \Vert \leq \frac{C}{2\pi N}\Vert U \Vert_{\CC^{d+1}} =:\sigma_N,\end{equation}
where $C$ depends only on $d$. Therefore
\begin{align*}
     \Vert \bar V U - I \Vert_{\CC^0} & \leq \Vert \bar V - U^{-1} \Vert_{\CC^0} \Vert U \Vert_{\CC^0} \\
        & \leq \Vert \bar V - \bar U \Vert_{\CC^0}\Vert U \Vert_{\CC^0} \quad \text{ from } U^{-1} = \bar U  \\
        & \leq \Vert V - U \Vert_{\CC^0} \Vert U \Vert_{\CC^0} \\
        & \leq \sigma_N \Vert U \Vert_{\CC^0} \\
        & \leq \frac{1}{2} \qquad (\ast_1)
\end{align*}
The last line of the equation is satisfied by assumption $\eqref{taillesigman}$ with $C_d \geq  \frac{C}{\pi}$. This implies that, for all $\theta$, $X(\theta)=\bar V(\theta) U(\theta)$ is invertible and

\[X^{-1}= (I+(\bar{V}U-I))^{-1}=\sum_{k\geq 0} (-1)^k(\bar{V}U-I)^k\]
hence
\[ \Vert X^{-1} -I \Vert_{\CC^0} 
\leq \sum_{k\geq 1} \Vert (\bar{V}U-I)\Vert_{\CC^0}^k
\leq  2 \Vert V - U \Vert_{\CC^0} \Vert U \Vert_{\CC^0}\leq 2\sigma_N \Vert U\Vert _{\mathcal{C}^0}. 
\]

\noindent
Therefore
\begin{align*}
 \bar V^{-1} & = U U^{-1} \bar V^{-1} = U X^{-1} = U + U(X^{-1}-I) \\
 & = \bar U ^{-1} + U(X^{-1} - I)  \quad \text{ since } \bar U = U^{-1} \\
 & = \bar X ^{-1} V + U(X^{-1} - I) \\
 & = V +(\bar X^{-1} - I)V + U(X^{-1}-I),
\end{align*}
and then
\[ V^{-1}-\bar V = ( X^{-1} - I)\bar V + \bar U( \bar X^{-1}-I)\]
which gives
\begin{align*}
 \Vert V^{-1} - \bar V \Vert_{\CC^0} & = \Vert ( X^{-1} - I)\bar V + \bar U( \bar X^{-1}-I) \Vert_{\CC^0} \\
    & \leq \Vert X^{-1} - I \Vert_{\CC^0} (\Vert V \Vert_{\CC^0} + \Vert U \Vert_{\CC^0}) \\
 & \leq 2\sigma_N \Vert U\Vert_{\mathcal{C}^0}(\Vert V \Vert_{\CC^0} + \Vert U \Vert_{\CC^0})\\
    & \leq 
4\sigma_N \Vert U \Vert_{\CC^0}( \Vert V \Vert_{\CC^0} + \frac{1}{2}\sigma_N) \\
    \end{align*}
(by definition of $\sigma_N$), therefore, using \eqref{taillesigman} with $C_d\geq \frac{16C}{\pi}$,
    \begin{equation}\label{ast_2}
 \Vert V^{-1} - \bar V \Vert_{\CC^0} 
     \leq \frac{1}{8} (\Vert V \Vert_{\CC^0} + \frac{\sigma_N}{2}) 
     \leq \frac{1}{8} \Vert V \Vert_{\CC^0} + \frac{1}{16} \sigma_N .\qquad  
\end{equation}

\noindent 
Now \eqref{taillesigman} also implies 

\[
\sigma_N\leq \frac{1}{32\Vert U\Vert_{\mathcal{C}^0}}
\]

\noindent and the property that $U^{-1}=\bar U$ implies that $\Vert U\Vert _{\mathcal{C}^0}\geq 1$, so $3\sigma_N\leq 2 \Vert U\Vert _{\mathcal{C}^0}$ also holds. This, together with \eqref{V-U-sigma_N}, implies that

\[
\sigma_N\leq 2(\Vert U\Vert _{\mathcal{C}^0}-\sigma_N)\leq 2(\Vert U\Vert _{\mathcal{C}^0} - \Vert V-U\Vert _{\mathcal{C}^0})\leq 2\Vert V\Vert _{\mathcal{C}^0}.
\]

\noindent
Thus \eqref{ast_2} implies that \eqref{conclusion-lemme-Vbar} holds. 
\end{proof}

The following proposition allows us to apply the previous lemma to the change of variables given by the almost reducibility, in order to conjugate the almost reducible cocycle to a cocycle arbitrarily close to a real one.

\begin{proposition}\label{proppresquered}
Let $N \in \N$, $\rho>0$.
Let $U : \T^d \rightarrow Gl(n, \C)$ of class $\CC^{d+1}$ such that
\[ U^{-1} = \bar U. \]
Let $B \in gl(n, \C)$ in Jordan normal form, 
and let
\[ G = \partial_\omega U - BU + U \bar B. \]
There exists a constant $C_1 \geq 0$ (depending only on $n,d$) such that if

\begin{enumerate}
\item  \begin{equation}\label{taille-sigman-prop}C_1 \Vert U \Vert_{\CC^{d+1}} \Vert U \Vert_{\CC^0}\leq N,  \end{equation}



\item 
\begin{equation}\label{smallness-prop}\Vert G \Vert_{\CC^{d+1}} 
\leq C_1^{-1}\rho^{2n-1}   N^{-(d-1)} \Vert U \Vert_{\CC^0}^{-n}, \end{equation}


\item 
\begin{equation}\label{kappagrand}
\rho \leq C_1^{-1}\min(\frac{1}{N^d \Vert U\Vert _{\mathcal{C}^0}^{n+1}},\frac{\kappa}{N^\tau})
,\end{equation}

\end{enumerate}
then there exist $W \in \CC^{\infty}(\T^d, Gl(n, \C))$, $B' \in gl(n, \C)$, and $B'' \in gl(n, \C)$ such that
\[ \partial_{\frac{\omega}{2}}W = BW - W(B'+B'')\]
where $B'$ is as in Lemma \ref{construction-de-W}.
Moreover, 
\[ \Vert B'' \Vert \leq 2n^2\rho ,\]
and
\[ \forall s \in \N,\ \Vert W^{\pm1} \Vert_{\CC^{s}} \leq (2\pi n N)^{2s} .\]
\end{proposition}

\begin{proof}
Denote $\xi = \Vert U \Vert_{\CC^0} \geq 1$ (the lower bound comes from $U^{-1} = \bar U$). Let \[V = \mathcal{R}_{N}U = \sum_{\vert k \vert \leq N} \hat U(k)e^{2i\pi \langle k,\theta\rangle}.\] We have
\[ \partial_\omega V = BV - V\bar B +F,\ \mathrm{where}\  \quad F = \mathcal{R}_N G.\]

\noindent 
We will now apply lemma \ref{construction-de-W} and for this, we need to find an upper bound of $\Vert V \Vert_{\CC^0}$ and an upper bound of $\Vert V^{-1} \Vert_{\CC^0}$. We have
\[ \Vert V - U \Vert_{\CC^{0}} \leq \sum _{\vert k \vert > N}\Vert \hat U(k) \Vert  \leq C_2 \frac{1}{ N}\Vert U \Vert_{\CC^{d+1}}\]
where $C_2\geq 1$ only depends on $d$. Suppose that $C_1 \geq 32C_2$ and that $N \geq  C_1 \Vert U \Vert_{\CC^{d+1}}\Vert U \Vert_{\CC^0}$,  then 

\[ \Vert V- U \Vert_{\CC^{0}}  
\leq C_2 \frac{1}{ N} \frac{ N}{ C_1 \Vert U \Vert_{\CC^0}} 
\leq \frac{1}{2\xi}.\]
Therefore since $\xi \geq 1$,
\[ \Vert V \Vert_{\CC^0} \leq \Vert U \Vert_{\CC^0} + \Vert V-U \Vert_{\CC^0}  \leq \xi + \frac{1}{2\xi} \leq 2\xi. \]

 \noindent 
 Let $C_d$ the constant given by lemma \ref{troncation}. If moreover $C_1 \geq C_d$, we have
 \[N \geq C_d \Vert U \Vert_{\CC^{d+1}}\Vert U \Vert_{\CC^0},\]
and we can apply lemma \ref{troncation} which gives
\[\Vert V^{-1}\Vert_{\CC^0} \leq \Vert V^{-1} - \bar V \Vert_{\CC^0} + \Vert \bar V \Vert_{\CC^0} \leq \frac{1}{4}\Vert V \Vert_{\CC^0} + \Vert V\Vert_{\CC^0} \leq 3\xi, \]
and the assumption \eqref{boundV} of lemma \ref{construction-de-W} is satisfied with $4\xi$ instead of $\xi$.
Now,
\[ \Vert F \Vert_{\CC^0} \leq \Vert G\Vert_{\CC^0}+ \Vert F-G \Vert_{\CC^0} \leq \Vert G \Vert_{\CC^0} + C_2 (\frac{1}{2\pi N}) \Vert G\Vert_{\CC^{d+1}}\]
\[ \leq  2 C_2\Vert G \Vert_{\CC^{d+1}} \leq 
2C_2C_1^{-1}\rho^{2n-1} \big(   N \Vert U \Vert_{\CC^0} \big)^{-dn^2} \]
so, if $C_1$ is large enough depending on $n,d$, the assumption \eqref{small2bis} of Lemma \ref{construction-de-W} holds. 


\noindent 
Therefore one can apply the Lemma \ref{construction-de-W} which directly gives the conclusions.
\end{proof}

\section{Proof of the main result}\label{section-proof-main-result}

\begin{theorem}
Let $X_{\omega, A}$ an almost reducible cocycle in $\CC^\infty$. If $X_{\omega, A}$ is real and if $\omega$ is Diophantine, then $X_{\frac{\omega}{2}, A_2}$ is real almost reducible in $\CC^\infty$, where $A_2(\theta) = A(2\theta)$ for all $\theta \in \T^d$.
\end{theorem}

\begin{proof}
By the proposition \ref{almost-jordan}, there exist $Z_j : \T^d \rightarrow Gl(n, \C)$ of class $\CC^\infty$, 
$B_j \in gl(n,\C)$ in Jordan normal form and $F_j \in C^\infty( \T^d , gl(n, \C))$ such that
\[ \forall j,\ \partial_\omega Z_j = A Z_j - Z_j(B_j+F_j)\]
\begin{equation} \label{convergence}  \quad \Vert Z_j ^{\pm 1} \Vert^m _{\CC^r}\Vert F_j \Vert_{\CC^r} \rightarrow_{j \rightarrow \infty} 0, \quad \forall r,m \in \N. \end{equation}
\bigskip
Fix $r \geq d+1$ and denote $\Vert F_j \Vert_{\CC^r} = \varepsilon_j$. By a similar reasoning as when proving \eqref{def-J1}, and from the convergence condition \eqref{convergence}, given $\beta' < \min(\frac{1}{2+8n^3\tau}, \frac{1}{16n^6})$, there exists $J_3$ such that, for all $j \geq J_3$,
\begin{equation} 
\Vert Z_j^{\pm 1} \Vert_{\CC^{r}} \leq \varepsilon_j^{-\beta'}. \label{def-J3}\end{equation}

\noindent Let $U_j = {Z_j}^{-1}\bar{Z_j}$. 
Then $U_j^{-1} = \bar U_j$ and 
\[ \xi_j = \Vert U_j \Vert_{\CC^0} = \Vert U_j^{-1}\Vert_{\CC^0} \geq 1,\]
and
\[ \Vert U_j \Vert_{\CC^{r}} \leq \varepsilon_j^{-2 \beta '}\]
Moreover
\begin{align*}
\partial_\omega U_j & = \partial_\omega({Z_j}^{-1}\bar Z_j)\\
			      & =  \partial_\omega ({Z_j}^{-1}) \bar Z_j + {Z_j}^{-1}\partial_\omega \bar{Z_j} \\
			      & = (( B_j +  F_j){Z_j}^{-1}-{Z_j}^{-1} A)\bar Z_j + {Z_j}^{-1} (\overline{AZ_j - Z_j( B_j+  F_j)}) \\
			      & =  B_j {Z_j}^{-1}\bar  Z_j +  F_j {Z_j}^{-1}\bar Z_j - {Z_j}^{-1}\bar Z_j \bar B_j - {Z_j}^{-1}\bar Z_j  \bar F_j +{Z_j}^{-1}(\bar A - A) \bar Z_j
\end{align*}
and since $A$ is real, 
\[ \partial_\omega U_j =  B_jU_j - U_j \bar B_j + G_j\]
where 
\[  G_j = F_j U_j - U_j \bar F_j\]
with, from inequality \eqref{1} given in introduction, 
\[ \Vert G_j \Vert_{\CC^r} \leq 2C_r\Vert U_j \Vert_{\CC^r} \Vert  F_j \Vert_{\CC^r}.\]
 Let
\[ \varepsilon'_j = 2C_r \Vert U_j \Vert_{\CC^{r}} \Vert  F_j \Vert_{\CC_r}\]
\[ N_j = C_1 \Vert U_j \Vert_{\CC^{d+1}} \Vert U_j \Vert_{\CC^0}\]
with $C_1$ the constant of proposition \ref{proppresquered}. We will apply Proposition \ref{proppresquered} with $N=N_j, \rho=\rho_j={\varepsilon'_j}^{\frac{1}{2n^3}}$, $U=U_j, B=B_j,G=G_j$. There are three assumptions to check. The assumption \eqref{taille-sigman-prop} holds by definition of $N$.

\noindent By definition of $\beta'$, we have
\[ \rho_j \Vert U_j\Vert_{\CC^0}^{n+1}(2N_j)^\tau \rightarrow_{j \rightarrow + \infty} 0,\]
and there exists $J_4$ (which we can choose $\geq J_3$) such that for all $j \geq J_4$,
\[ \kappa > C_1\rho_j\Vert U_j\Vert_{\CC^0}^{n+1} (2N_j)^\tau,\]
therefore the assumption \eqref{kappagrand} of Proposition \ref{proppresquered} holds.
Moreover, 
by definition of $N_j$ and $\varepsilon_j'$,
\begin{align*}
\rho_j^{-(2n-1)}( N_j  \Vert U_j \Vert_{\CC^0})^{dn^2}\varepsilon'_j & 
\leq( \varepsilon_j')^{1-\frac{2n-1}{2n^3}}(C_1\varepsilon_j^{-6\beta'})^{dn^2}\\
& \leq C (\varepsilon_j^{-2\beta'})^{1-\frac{2n-1}{2n^3}+3dn^2}\varepsilon_j^{1-\frac{2n-1}{2n^3}}
\end{align*}
where $C$ is a constant depending only on $d,n,r$, and the right hand side tends to $0$ as $j$ tends to infinity
. Therefore there exists $J_5$ ($\geq J_4$) such that, for $j \geq J_5$,
\[ \Vert G_j \Vert_{\CC^r} \leq \varepsilon'_j \leq  C_1^{-1}\rho_j^{2n-1}(N_j \Vert U_j \Vert_{\CC^0})^{-dn^2}\]
therefore the assumption \eqref{smallness-prop} of Proposition \ref{proppresquered} holds.
Apply the proposition \ref{proppresquered}: 
there exist $W_j : \T^d \rightarrow Gl(n, \C)$ of class $\CC^\infty$, $B'_j \in gl(n,\R)$ and $B''_j\in gl(n, \C)$ such that $B'_j$ 
has the same block diagonal structure as $B_j$, 
with
\[ \partial_\frac{\omega}{2}W_j =  B_jW_j - W_j(B'_j+B''_j),\]
\[ \Vert B''_j \Vert \leq 2n\rho = 2n{\varepsilon'}_j^{\frac{1}{2n^3}} \leq C_{n,r}\Vert U_j \Vert_{\CC^r}^{\frac{1}{2n^3}} \varepsilon_j^{\frac{1}{2n^3}},\]
\begin{equation} \Vert W_j^{\pm1} \Vert_{\CC^s} \leq (2 \pi nN_j)^{2s}, \quad \forall s \in \N, \label{estimationW}\end{equation}\\
Moreover, denoting $G_j = B_j''+W_j^{-1} F_j W_j$, we get
\begin{align*}
    \Vert G_j \Vert_{\CC_r} & \leq C_{n,r}\Vert U_j \Vert_{\CC^r}^{\frac{1}{2n^3}}\varepsilon_j^{\frac{1}{2n^3}}+ (2\pi nN_j)^{4r}\varepsilon_j \\
    & \leq C_{n,r}\Vert U_j \Vert_{\CC^r}^{\frac{1}{2n^3}}\varepsilon_j^{\frac{1}{2n^3}} + (32 C_1n)^{4r} \Vert U_j \Vert_{\CC^{d+1}}^{4r} \Vert U_j \Vert_{\CC^0}^{4r} \varepsilon_j  \\
    & \leq C_{r,n}' \Vert U_j \Vert_{\CC^r}^{8r} \varepsilon_j^{\frac{1}{2n^3}} .
\end{align*} 


Finally, for our fixed $m$ and $r \geq d+1$,
\begin{align*}
\Vert  (Z_jW_j)^{\pm 1}\Vert_{\CC^r}^m \Vert G_j \Vert_{\CC^r} &\leq  C_{r,n}''\Vert W_j^{\pm 1}\Vert_{\mathcal{C}^r} ^m \Vert Z_j^{\pm 1} \Vert_{\CC^r}^m \Vert U_j \Vert_{\CC^r}^{8r}\varepsilon^{\frac{1}{2n^3}} \\
& \leq C_{r,n}''' N_j^{4rm} \Vert Z_j^{\pm 1} \Vert_{\CC^r}^m\Vert U_j \Vert_{\CC^r}^{8r}\varepsilon^{\frac{1}{2n^3}} \\
& \leq C_{r,n}'''' \Vert U_j  \Vert_{\CC^r}^{8rm + 8r} \Vert Z_j^{\pm 1} \Vert_{\CC^r}^m\varepsilon^{\frac{1}{2n^3}} \\
& \leq C_{r,n}''''' \Vert Z_j^{\pm 1} \Vert_{\CC^r}^{16rm +16r+m}\varepsilon^{\frac{1}{2n^3}}
\end{align*} 
which tends to $0$ as $j \rightarrow \infty$ by assumption \eqref{convergence}. 

\bigskip
\noindent This construction depends on $m,r$; however, applying Lemma \ref{convergence-m}, one gets almost reducibility to a sequence of real matrices.

\bigskip
\noindent Notice that the transformations $ Z_jW_j$ are not real, but we can now apply Proposition \ref{presqueconjugaisonreelle} to get real ones.
\end{proof}

\section{Appendix}\label{blocsproches}

\subsection{A small divisors lemma}

The following classical lemma is useful to control the small divisors which might occur:

\begin{lemma}\label{majoration}
Let $f : \T^d \rightarrow \C$ of class $\CC^\infty$ and $N \in \N^{\ast}$ such that \[ \hat f(k) = 0, \quad \vert k \vert > N.\]
Let $\rho>0$. If $\alpha \in \C$ is such that 
\[ \vert \alpha - 2i\pi \langle k , \omega \rangle \vert \geq  \rho, \quad \forall \vert k \vert \leq N,\]
then the equation

\[
\left\{\begin{array}{l}
\partial_\omega u(\theta) = \alpha u(\theta) + f(\theta), \quad \theta \in \T^d, \\
\hat u(k)=0, \quad \text{if } \vert k \vert > N
\end{array}\right. \tag{$\ast$}
\]
has a unique solution $u: \T^d \rightarrow \C$ of class $\CC^\infty$ and it satisfies
\[ \Vert u \Vert_{\CC^0} \leq C \rho^{-1} N^{\frac{d+1}{2}}\Vert f\Vert_{\CC^0}\]
where $C$ is a constant depending on $d$.
\end{lemma}

\begin{proof}
Decomposing the first equation $(\ast)$ into Fourier coefficients, we have
\[ 2i\pi \langle k ,\omega \rangle \hat u (k) = \alpha \hat u(k) + \hat f(k), \quad k\in \Z^d\]
and the unique solution of $(\ast)$ is defined by
\[
\hat u (k) = 
\left\{\begin{array}{c}
\frac{1}{2i \pi \langle k, \omega \rangle - \alpha } \hat f(k), \quad \forall \vert k \vert \leq N \\
 0, \quad \forall \vert k \vert > N.
\end{array}\right.
\]
By H\"older's inequality,

\[ \Vert u \Vert_{\CC^0} \leq \sum_k \vert \hat u(k) \vert \leq \frac{1}{\rho}\sum_{\vert k \vert \leq N} \vert \hat f(k) \vert \]
\[\leq \frac{1}{\rho}\sqrt{\sum_{\vert k \vert \leq N}1} \sqrt{\sum_{\vert k \vert \leq N} \vert \hat f(k) \vert^2} \leq C \frac{1}{\rho}N^{\frac{d+1}{2}} \Vert f \Vert_{L^2} \leq C \frac{1}{\rho}N^{\frac{d+1}{2}}\Vert f\Vert_{\CC^0},\]
where $C$ depends only on $d$.
\end{proof}

\subsection{Separating the spectrum of a matrix}

\begin{definition}[$\Gamma$-separation, $\Gamma$-connection]\label{def-gamma-close}
     Let $E_1$, $E_2$ two finite sets of complex numbers and let $\Gamma > 0$. We say that $E_1$ and $E_2$ are $\Gamma$-separated if for all $\alpha \in E_1$ and for all $\beta \in E_2$, $|\alpha-\beta|>\Gamma$. 

     We say that $E_1=\{\alpha_1, \dots, \alpha_n\}$ is $\Gamma$-connected if there is no decomposition $E_1=E'_1\cup E''_1$ with $E'_1$ and $E''_1$ $\Gamma$-separated. 
\end{definition}

\begin{remark}  If $E_1$ is $\Gamma$-connected, then for all $\alpha,\beta\in E_1$, $|\alpha-\beta|\leq \# E_1\cdot \Gamma$.
\end{remark}

\begin{lemma}\label{conjugaison-blocs} Let $B\in \mathcal{M}(n, \C)$ and $\Gamma > 0$. There exists $\tilde M\in Gl(n, \C)$ such that $\tilde M^{-1}B\tilde M$ is block diagonal, with each block being upper triangular with $\Gamma$-connected spectrum, and such that the spectrum of two distinct blocks are $\Gamma$-separated. Moreover $\tilde{M}$ satisfies 

\[ \Vert \tilde M \Vert \leq n^{3n} (\frac{\Vert B \Vert}{\Gamma^2})^{n^3} \]
\[ \Vert \tilde M^{-1} \Vert \leq n^{3n} (\frac{\Vert B \Vert}{\Gamma^2})^{n^3}, \]
and
\[ \Vert \tilde M^{-1}B \tilde M \Vert \leq C(n)B\]
where $C(n)$ is a constant depending on $n$.

\end{lemma}

\begin{proof}
\textbf{Step 1:} By Schur decomposition, there exist an unitary matrix $Q \in Gl(n, \C)$ and an upper triangular matrix $T\in gl(n,\C)$ such that
\[ B = QTQ^{-1}\]

\noindent 
Obviously $\Vert B\Vert =\Vert T\Vert$.

\bigskip
\textbf{Step 2:} Let $E_1, \dots, E_l$ be $\Gamma$-separated sets of eigenvalues of $B$, each of them being $\Gamma$-connected, and write (up to a permutation, which does not change the estimates) 
\[ B = \left(\begin{array}{cc}T_1 & T_2  \\0 & T_4\end{array}\right), \quad D = \left(\begin{array}{cc}T_1 & 0  \\0 & T_4\end{array}\right)\]
where $T_1$ is the block corresponding to the eigenvalues of $E_1$ and $T_4$ is the block corresponding to the eigenvalues of $E_2 \cup \dots \cup E_l$.
Write

\[ M = \left(\begin{array}{cc}I & R  \\0 & I\end{array}\right), M^{-1} = \left(\begin{array}{cc}I & -R  \\0 & I\end{array}\right),  \]
then the conjugation $MBM^{-1}=D$ is equivalent to
\[ T_1 R + T_2 = RT_4\]

\noindent where $R$ is the unknown.
Let $n_1$ be the dimension of $T_1$ and $n_2$ the dimension of $T_4$. 
Decomposing to coefficients, the previous matrix equation can be written
\[ -(T_2)_{i,j}= \sum_{k= i}^{n_1} (T_1)_{i,k}R_{k,j}-\sum_{l=1}^jR_{i,l}(T_4)_{l,j}\]

\noindent 
Solving these equations in the following order: \[(i,j)=(n_1,1),(n_1,2),\dots, (n_1,n_2),(n_1-1,1),\dots, (n_1-1,n_2),\dots, (1,1),(1,2),\dots,(1,n_2),\]
one sees that the coefficients of $R$ have upper bound $ (n_1n_2)(\frac{\Vert B\Vert}{\Gamma^2})^{n_1n_2}$ which implies that 

$$\Vert M\Vert\leq n(n_1n_2)(\frac{\Vert B\Vert}{\Gamma^2})^{n_1n_2}.$$

\noindent 
Iterate step 2 replacing $B$ by $T_4$ and $T_4$ by the blocks corresponding to $E_{l-k}\cup\dots \cup E_l$.
The algorithm stops when $T_4$ only has the eigenvalues of $E_l$, that is to say after at most $n-1$ steps.

\bigskip
\noindent 
Finally there exist an invertible matrix $\tilde{M}$ of dimension $n$, a block diagonal matrix $\tilde D=(\tilde D_1, \dots, \tilde D_l)$ of dimension $n$, with each block $T_i$ corresponding to the group of eigenvalues $E_i$ (which remains $\Gamma$-separated). 

\[\tilde D = \tilde{M}B\tilde{M}^{-1}\] with \[\Vert \\\tilde{M}\Vert,\Vert\bar{M}^{-1}\Vert \leq  n^{3n}(\frac{\Vert B\Vert}{\Gamma^2})^{n^3}. \]

\noindent 
Since $\tilde D$ is constructed from $B$ by removing coefficients, its norm is $\leq C(n) \Vert B \Vert$ where $C(n)$ is a constant depending on $n$.
\end{proof}

\begin{corollary}\label{induction-blocs}
    
    Let $A\in gl(n, \C)$. Given a positive decreasing sequence $(\Gamma_i)$, there exists $d_0 \in \N$, $d_0\leq n$, depending on $A$, and there exists $B \in gl(n, \C)$ a block diagonal matrix where each block is upper-triangular with $\Gamma_{d_0}$-connected spectrum, and $S$ invertible such that
    \[ A = SB S^{-1}\]
\[ \Vert S \Vert, \Vert S^{-1} \Vert  
\leq n^{3n} (\frac{\Vert B \Vert}{\Gamma_{d_0-1}^2})^{n^3},\]
and
\[ \Vert A \Vert \leq C(n) \Vert B \Vert\]
where $C(n)$ is a constant depending on $n$.
\end{corollary}

\begin{proof}
    By induction on $\Gamma_i$:
\textbf{Base case:} first apply lemma \ref{conjugaison-blocs} on $A$ with $\Gamma = \Gamma_0$. This gives a conjugation $S$ from $A$ to a matrix $B$ where each block is upper triangular with $\Gamma_0$-connected spectrum. If the spectrum of every block of $B$ is also $\Gamma_1$-connected, then we are done, if not, apply lemma \ref{conjugaison-blocs} to $A$ with $\Gamma_1$. \\
\textbf{Induction step:} given $d \in \N^{\ast}$, apply lemma \ref{conjugaison-blocs} on $A$ with $\Gamma_d$. This gives a conjugation $S$ from $A$ to a block diagonal matrix $B$ where each block is upper triangular with $\Gamma_d$-connected spectrum. If the spectrum of every block of $B$ is also $\Gamma_{d+1}$-connected, then we are done, if not, apply lemma \eqref{conjugaison-blocs} to $A$ with $\Gamma_{d+1}$. \\

After each step, the number of eigenvalues in each block decreases (because you split a block at each step). Then there exists an index $d_0$ such that applying lemma \eqref{conjugaison-blocs} on $A$ with $\Gamma_{d_0-1}$, this gives $B$ and $S$ where the spectrum of every block of $B$ is $\Gamma_{d_0}$-connected and the matrix $S$ satisfies the estimate with $\Gamma_{d_0-1}$. 
\end{proof}
\begin{remark}
If $\# \sigma(A) = n_0$, then if the induction is applied $n_0$ times, each block contains only one eigenvalue. Also notice that the induction does not change $A$, whence the exponents in the estimate.
\end{remark}

\subsection{Stability of the Jordan structure}

Let $A=(a_i^j)_{1\le i,j\le n}$ and $B=(b_i^j)_{1\le i,j\le n}$ be two matrices with entries $0$ and $1$ such that for some $r\in \{1,\dots, n-1\}$,
$$a_i^j=b_i^j=0\quad \mathrm{if}\  j-i\not=r.$$

 \medskip
 
\begin{lemma}\label{rank} Let $\varepsilon, \xi > 0$. If there exists an invertible matrix $C$ such that
$$\Vert C\Vert ,\Vert C^{-1}\Vert \le \xi$$
and
$$\Vert AC-CB\Vert=\varepsilon < \frac1{n!\xi^n},$$
then $\mathrm{rank}\ A=\mathrm{rank}\ B$.
\end{lemma}

\medskip

\begin{proof} Let
$$I=\{ i\in \llbracket 1,n-r \rrbracket : a_i^{i+r}=1\}\quad\&\quad J=\{ j\in\llbracket r+1,n \rrbracket : b_{j-r}^{j}=1\}.$$
Then
$$(AC-CB)_i^j=a_i^{i+r}c_{i+r}^j- c_i^{j-r} b_{j-r}^j=
\begin{cases}
c_{i+r}^j & \mathrm{if} \ i\in I\ \text{ and } j\notin J \\
-c_i^{j-r} & \mathrm{if} \ i\notin I\ \text{ and } j\in J.
\end{cases}
$$
thus 
$$
\begin{cases}
\vert c_{i+r}^j\vert \le \varepsilon  & \mathrm{if} \ i\in I\ \text{ and } j\notin J \\
\vert c_i^{j-r}\vert\le \varepsilon  & \mathrm{if} \ i\notin I\ \text{ and } j\in J.
\end{cases} $$

\noindent 
It follows that 
if $\#I\neq \#J$,
$$\vert \det C\vert\le n!\varepsilon\xi^{n-1}.$$

\noindent 
Now
$$\Vert C^{-1}A-BC^{-1} \Vert=\Vert C^{-1}\big( AC-CB\big) C^{-1} \Vert \leq  \varepsilon \xi^2$$
so, in the same way as for $C$, 
$$\vert \det C^{-1}\vert\le n!\varepsilon\xi^2\xi^{n-1}$$
if $\#I$ is different from $ \# J$.

\noindent 
Therefore
$$1=\vert \det(C C^{-1})\vert\le (n!\xi^n\varepsilon)^2$$
which is impossible by assumption. Hence $\#I=\# J$.
\end{proof}

\medskip

\noindent 
Let now $A$ and $B$ be two nilpotent $n\times n$-matrices on Jordan normal form.

\begin{proposition}\label{pApp2} If there exists an invertible matrix $C$,
$$\Vert C\Vert ,\Vert C^{-1}\Vert \le \xi$$
such that
$$\Vert AC-CB\Vert=\varepsilon < \frac1{n\cdot n!\xi^n},$$
then $A$ and $B$ have the same Jordan structure, i.e. they have the same number of Jordan blocks
of dimension $k$, for all $k=1,2,\dots$.
\end{proposition}

\medskip

\begin{proof} Let
$$X_k=A^kC-C B^k.$$
Since $\Vert X_1\Vert\le \varepsilon$,
it follows by an easy induction that
$$\Vert X_k\Vert\le k\varepsilon\le (n-1)\varepsilon:$$

\begin{itemize}
\item For $k=1$, this is the assumption;
\item For $k\geq 1$, notice
$$X_{k+1}=A^{k+1}C -CB^{k+1}= A^k AC-CB^{k+1} = A^k (CB+X_1) -CB^k B =X_k B + A^k X_1
$$
and the estimate follows by induction since $A,B$ have norm $\leq 1$.
\end{itemize}

\noindent 
Then, by the lemma \ref{rank},
$$\mathrm{rank}\ A^k=\mathrm{rank}\ B^k,\quad \forall k\geq 1$$
This implies the statement.

\end{proof}

\subsection{Conjugation to a real matrix whose spectrum is stable by complex conjugation}

\begin{lemma}\label{spectre-conjugué}
Let $A$ a matrix in  Jordan normal form, such that if $\alpha$ is an eigenvalue, then $\bar \alpha$ is also an eigenvalue and the Jordan blocks of $\alpha$ and $\bar \alpha$ are identical. 
There exists a unitary matrix $P$ 
such that $P^*AP$ is on real Jordan normal form.
\end{lemma}

\begin{proof}

It suffices to prove this for $A=\left(\begin{array}{cc}
\alpha I+N & 0\\
0 & \bar \alpha I +N\\
\end{array}\right)$, where $N$ is a nilpotent Jordan block (of dimension $n\times n$). After the permutation $(e_1,e_2,\dots, e_{2n})\mapsto (e_1,e_{n+1},e_2,e_{n+2},\dots, e_{2n-1},e_{2n})$, the matrix $A$ takes the block triangular form $A=(A_i^j)_{i,j=1,\dots, n}$ with

$$A_i^j= \left\{\begin{array}{l}
U\ \mathrm{if} \ j=i\\
I \ \mathrm{if} \ j-i=1\\
0 \ \mathrm{otherwise}\\
\end{array}\right.
$$

\noindent and $U= \left(\begin{array}{cc}
\alpha & 0\\
0 & \bar \alpha\\
\end{array}\right)$. Let $C=\frac{1}{\sqrt{2}}\left(\begin{array}{cc}
1 & -i\\
1 & i\\
\end{array}\right)$ (which is unitary) and let $P=\operatorname{diag} (P_j)_{j=1,\dots , n},\ P_j=C$. Then 
$P^*AP=(B_i^j)_{i,j=1,\dots , n}$
with 
$B_i^j = \left(\begin{array}{cc}
\operatorname{Re}\alpha & \operatorname{Im} \alpha  \\
-\operatorname{Im}\alpha & \operatorname{Re}\alpha\\
\end{array}\right)$ if $i=j$, $B_i^j = I$ if $j-i=1$, and $B_i^j=0$ otherwise.
\end{proof}

\subsection{A lemma about almost reducibility}

\begin{lemma}\label{convergence-m}
    Assume for all $m' \in \N$, there exist $(Z_{j,m'})_j $ a sequence of $\CC^{\infty}$ maps defined on $\T^d$ and $(F_{j,m'})_j$ a sequence of $\CC^{\infty}$ maps defined on $\T^d$ such that
    \[ \Vert Z_{j,m'}^{\pm 1} \Vert_{\CC^{m'}}^{m'} \Vert F_{j,m'} \Vert_{\CC^{m'}} \rightarrow_{j \rightarrow + \infty} 0,\]
    then there exists $(\tilde Z_j)$ and $(\tilde F_j)$, subsequences of $(Z_{j,m'})_j $ and $(F_{j,m'})_j$ extracted from the same indices, such that for all $m,r \in \N$, 
     \[ \Vert \tilde Z_{j}^{\pm 1} \Vert_{\CC^r}^{m} \Vert \tilde F_{j} \Vert_{\CC^r} \rightarrow_{j \rightarrow + \infty} 0.\]
\end{lemma}

\begin{proof}
For all $m\in\N$, let $j_{m} \in \N$ such that, for all $j \geq j_{m}$,
\[\Vert {Z}_{j_{m},m}^{\pm 1} \Vert^{m} _{\CC^{m}}\Vert F_{j_{m},m} \Vert_{\CC^{m}} \leq \frac{1}{m}\]
and such that the sequence $(j_{m})_{m\in \N}$ is strictly increasing. Then, for all $j$ large enough, there exists $m := m(j) \in \N$ such that $j_m \leq j < j_{m+1}$. Denote then, for all $j$ large enough
\[ \tilde F_j = F_{j,m(j)},\ \ \tilde Z_j = Z_{j, m(j)}.\]

\noindent 
Let $\tilde m \in \N$. Then for all $j$ large enough such that $m(j) > \tilde m$,
\[ \Vert \tilde {Z}_j^{\pm 1} \Vert_{\CC^{\tilde m}}^{\tilde m} \Vert \tilde F_j \Vert_{\CC^{\tilde m}} \stackrel{def}{=} \Vert  {Z}_{j,m(j)}^{\pm 1} \Vert_{\CC^{\tilde m}}^{\tilde m} \Vert F_{j,m(j)} \Vert_{\CC^{\tilde m}}\leq \Vert  {Z}_{j,m(j)}^{\pm 1} \Vert_{\CC^{\tilde m}}^{m(j)} \Vert  F_{j,m(j)} \Vert_{\CC^{\tilde m}} \leq \frac{1}{m(j)}.\]
Since the constructed sequence $(j_m)_m$ is increasing, the sequence $(m(j))_j$ is also increasing, and then 

\begin{equation}\label{conv1}\Vert \tilde {Z_j}^{\pm 1} \Vert_{\CC^{\tilde m}}^{\tilde m} \Vert \tilde F_j \Vert_{\CC^{\tilde m}} \rightarrow 0,\quad j\rightarrow \infty.\end{equation}

\noindent 
Now let $r, m\in \mathbb{N}$.
Since \eqref{conv1} holds for all $\tilde m$, then in particular for $\tilde m = \max(r,m)$, 
 \[ \Vert \tilde Z_j^{\pm 1} \Vert_{\CC^{\tilde m}}^{\tilde m} \Vert \tilde F_j \Vert_{\CC^{\tilde m}} \rightarrow 0 \]
\noindent which implies the convergence condition is satisfied.

\end{proof}


\end{document}